\documentclass[a4paper,oneside,10pt,,mathscr,final,reqno]{amsart}

\usepackage{amssymb}
\usepackage{euscript}
\usepackage[frame,cmtip,curve,arrow,matrix,line,graph]{xy}

\newcommand{\bb} {\mathbb}

\newcommand{\cal}{\mathcal}
\newcommand{\frk}{\mathfrak}

\newcommand{\wh}{\widehat}

\newcommand{\ep}{\epsilon}

\newcommand{\e}{\mathrm{e}}
\newcommand{\red}{\mathrm{red}}
\newcommand{\tw}{\mathrm{tw}}

\newcommand{\Id}{\mathrm{Id}}
\newcommand{\Rep}{\mathrm{Rep}}

\newcommand{\Aut}{\operatorname{Aut}}

\newcommand{\Coker}{\operatorname{Coker}}
\newcommand{\Ext}{\operatorname{Ext}}
\newcommand{\Ho} {\operatorname{Ho}}
\newcommand{\Hom}{\operatorname{Hom}}

\newcommand{\Iso}{\operatorname{Iso}}
\newcommand{\Ker}{\operatorname{Ker}}
\newcommand{\Obj}{\operatorname{Obj}}

\newcommand{\blob}{{\scriptscriptstyle \bullet}}
\newcommand{\cl}{\widehat}
\newcommand{\longinto}{\lhook\joinrel\longrightarrow}
\newcommand{\ctp}{\mathop{\wh{\otimes}}}

\newcommand{\<}{\langle}
\renewcommand{\>}{\rangle}

\newcommand{\bu}[1]{#1_{\blob}}
\newcommand{\cpx}[4]{
 \begin{xy}
  \xymatrix@=0pt@C=20pt@R=15pt{#1
  \ar@<.5ex>[r]^{#2} & \ar@<.5ex>[l]^{#3}  #4}
 \end{xy}  
}

\theoremstyle{plain}
 \newtheorem{thm}{Theorem}[section]
 \newtheorem{lem}[thm]{Lemma}
 \newtheorem{prop}[thm]{Proposition}
 
 \newtheorem*{thm*}{Theorem}
\theoremstyle{definition}
 \newtheorem{dfn}[thm]{Definition}
 \newtheorem{fct}[thm]{Fact}
 \newtheorem{df}[thm]{Definition/Fact}
 \newtheorem{rmk}[thm]{Remark}
\theoremstyle{remark}

\numberwithin{equation}{section}
\allowdisplaybreaks

\begin{document}

%%%%% Authors & Paper data %%%%%

\title{Bialgebra structure on Bridgeland's Hall algebra of two-periodic complexes}
\author{Shintarou Yanagida}
\address{Research Institute for Mathematical Sciences,
Kyoto University, Kyoto 606-8502, Japan}
\email{yanagida@kurims.kyoto-u.ac.jp}

\thanks{The author is supported by JSPS Fellowships 
for Young Scientists (No. 24-4759).
This work is also partially supported by the 
JSPS Strategic Young Researcher 
Overseas Visits Program for Accelerating Brain Circulation
``Deepening and Evolution of Mathematics and Physics,
Building of International Network Hub based on OCAMI''}

\date{April 25, 2013}

%%%%% front matter %%%%%

\begin{abstract}
We study the bialgebra structure of 
the Hall algebra of two-periodic complexes 
recently introduced by Bridgeland.
We introduce coproduct on Bridgeland's Hall algebra,
and show that 
in the hereditary case 
the resulting bialgebra structure coincides 
with that on  Drinfeld double of the ordinary Hall algebra 
\end{abstract}

\maketitle
\tableofcontents

%%%%%%%%%%%%%%%%%%%%%%%%%%%%%%%%%%%%%%%%%%%%%%%%%%
%%%%%%%%%%%%%%%%%%%%%%%%%%%%%%%%%%%%%%%%%%%%%%%%%%
%%%%%%%%%%%%%%%%%%%%%%%%%%%%%%%%%%%%%%%%%%%%%%%%%%

\section{Introduction}

%%%%%%%%%%%%%%%%%%%%%%%%%%%%%%%%%%%%%%%%%%%%%%%%%%
%%%%%%%%%%%%%%%%%%%%%%%%%%%%%%%%%%%%%%%%%%%%%%%%%%

\subsection{}

This paper is a sequel to our previous paper \cite{Y},
where we studied the Hall algebra of 
$\bb{Z}/2\bb{Z}$-graded complexes
introduced by Bridgeland \cite{B}.
For any abelian category $\cal{A}$
over a finite field $\frk{k} := \bb{F}_q$
with finite dimensional morphism spaces, 
one can consider the Hall algebra $\cal{H}(\cal{A})$,
which was introduced by Ringel \cite{R1},
and its twisted version $\cal{H}_{\tw}(\cal{A})$.
Let us denote by 
$\cal{P} \subset \cal{A}$ be the subcategory of projective objects.
The algebra $\cal{DH}(\cal{A})$ Bridgeland introduced 
is a localization of $\cal{H}_{\tw}(\cal{C}(\cal{P}))$ 
by the set of acyclic complexes:
\begin{align}
\label{eq:B:embedding}
\cal{DH}(\cal{A}) := 
\cal{H}_{\tw}(\cal{C}(\cal{P}))
\bigl[ [\bu{M}]^{-1} \mid H_*(\bu{M})=0\bigr].
\end{align}
Here $\cal{C}(\cal{P}) \equiv \cal{C}_{\bb{Z}/2\bb{Z}}(\cal{P})$ is 
the category of $\bb{Z}/2\bb{Z}$-graded complexes 
of objects in $\cal{P}$,
which is a subcategory of the abelian category 
$\cal{C}(\cal{A}) \equiv \cal{C}_{\bb{Z}/2\bb{Z}}(\cal{A})$
of $\bb{Z}/2\bb{Z}$-graded complexes of objects in $\cal{A}$.
$\cal{H}_{\tw}(\cal{C}(\cal{P}))$ means a subalgebra of 
the twisted Hall algebra $\cal{H}_{\tw}(\cal{C}(\cal{A}))$
generated by $\cal{C}(\cal{P})$.

In \cite{B}, it was shown that 
$\cal{DH}(\cal{A})$ is an associative algebra with unit,
and that it has a nice basis if $\cal{A}$ is hereditary. 
As a result, Bridgeland was able to show \cite[Theorem~1.2]{B} that 
one has an embedding of algebras
$$
 U_t(\frk{g}_Q) \longinto \cal{DH}_{\red}(\Rep_{\frk{k}}(Q)).
$$
Here $\cal{A}=\Rep_{\frk{k}}(Q)$ 
is the category of finite-dimensional 
$\frak{k}$-representations of a finite quiver $Q$ 
without oriented cycles,
and $\cal{DH}_{\red}(\cal{A})$ is a certain quotient 
of $\cal{DH}(\cal{A})$.
$\frk{g}_Q$ is the derived Kac-Moody Lie algebra 
associated to $Q$,
and $U_t(\frk{g}_Q)$ is the quantized enveloping algebra of $\frk{g}_Q$
with $t=\sqrt{q}$ a fixed square root of $q = \# \frak{k}$.

The result \eqref{eq:B:embedding} 
is a natural extension of the classical result due to 
Ringel \cite{R1} and Green \cite{Gr},
where the upper half (or the Borel part) 
of the whole quantum group is realized 
as a subalgebra of the ordinary (twisted) Hall algebra.
\\

In the previous paper \cite{Y},
we showed that for a general hereditary category $\cal{A}$
the algebra $\cal{DH}(\cal{A})$ 
is, as an associative algebra, 
isomorphic to the Drinfeld double \cite{D}
of the extended Hall bialgebra $\cal{H}^{\e}_{\tw}(\cal{A})$
as an associative algebra.
Let us mention that 
this claim was stated without proof in \cite[Theorem~1.2]{B}.
Combining it with the result of Cramer \cite{C},
one can prove the invariance of $\cal{DH}(\cal{A})$ 
under the derived equivalence of $\cal{A}$.
Let us also mention that in a recent work of Gorsky \cite{Go}
more abstract approach is taken to prove the derived equivalence.
\\

A natural question on these results is 
whether one can introduce a coproduct and a bialgebra structure 
on $\cal{DH}(\cal{A})$.
In this paper we show that there exists a coproduct 
on $\cal{DH}(\cal{A})$
which gives a bialgebra structure.
Our coproduct is, as a result, a natural analogue of that 
on the ordinary Hall algebra introduced by Green \cite{Gr}.
We recommend \cite[\S~1.5]{S} and \cite[Part I, II]{R2}
for detailed and nice reviews of Green's coproduct 
and the proof of its bialgebra property.
\\

However, our coproduct is in some sense artificial.
One of the obstructions of constructing coproduct 
is that Bridgeland's Hall algebra is defined as a localization of 
(non-commutative) algebra.
Although one can define a coproduct on the unlocalized algebra 
by taking a straightforward analogue of Green's coproduct,
it doesn't descend to the localized algebra.
See \S\ref{subsec:naive} for a detailed explanation.

Our strategy to handle this obstruction is 
substituting the set of exact sequences 
in the category of complexes by a restricted one,
and making the structure constants for coproduct smaller. 
To realize this idea, we use the notion of exact category 
in the sense of Quillen \cite{Q}.
By Hubery \cite{H} one can construct Hall algebra from 
an exact category, and the resulting algebra is indeed 
a unital associative algebra.
We will introduce an exact category structure 
on the category of two-periodic complexes,
and from the resulting Hall algebra for an exact category,
we make a coassociative coproduct on Bridgeland's Hall algebra.
See \S\ref{subsec:exact} and \S\ref{subsec:genuine} for the detail.
Let us mention that our strategy is inspired 
by the recent work \cite{Go}.

In \S\ref{sect:hered} we treat the case where $\cal{A}$ is hereditary. 
In this case our coalgebra structure is compatible 
to the embedding $\cal{H}(\cal{A}) \hookrightarrow \cal{DA}(\cal{H})$ 
of the ordinary Hall algebra into Bridgeland's Hall algebra.
In this sense, our coproduct is a natural analogue of Green's coproduct.

Let us close this introduction by indicating further directions.
The first direction is investigating Hopf algebra structure.
In \cite{X}, Xiao introduced a (topological) Hopf algebra structure 
on the ordinary Hall algebra.
We expect a similar Hopf algebra structure can be introduced 
on the algebra $\cal{DA}(\cal{H})$.
Another direction is the investigation of higher dimensional case.
As mentioned in \cite{B}, $\cal{DA}(\cal{H})$ makes sense 
for abelian category of arbitrary global dimension,
but it looks too large.
Our construction of coproduct might suggest 
that in order to obtain a moderate algebra,
it is better to consider a restriction on the counting of extensions.
We expect that 
stability conditions on  triangulated categories,
which was also introduced by Bridgeland \cite{B2},
is related to this direction.

%%%%%%%%%%%%%%%%%%%%%%%%%%%%%%%%%%%%%%%%%%%%%%%%%%
%%%%%%%%%%%%%%%%%%%%%%%%%%%%%%%%%%%%%%%%%%%%%%%%%%

\subsection{Notations and conventions}
\label{subsect:notation}

Here we indicate several global notations.

$\frk{k} := \bb{F}_q$ is a fixed finite field 
unless otherwise stated,
and all the categories will be $\frk{k}$-linear.
We choose and fix a square root $t:=\sqrt{q}$.

For an abelian category $\cal{A}$, 
we denote by $\Obj(\cal{A})$ the class of objects of $\cal{A}$.
For an object $M$ of $\cal{A}$, 
the class of $M$ in the Grothendieck group $K(\cal{A})$ 
is denoted by $\cl{M}$.
The subcategory of $\cal{A}$ consisting of projective objects 
is denoted by $\cal{P}$.

In our argument we impose several assumptions 
on an abelian category $\cal{A}$.
Let us introduce the following conditions:
\begin{enumerate}
\item[(a)] essentially ($=$ skeletally) small with finite morphism spaces,
\item[(b)] linear over $\frk{k}$,
\item[(c)] of finite global dimension and having enough projectives.
\item[(d)] $\cal{A}$ is hereditary, that is of global dimension at most 1,
\item[(e)] nonzero objects in $\cal{A}$ define nonzero classes in $K(\cal{A})$.
\end{enumerate}

For an abelian category $\cal{A}$ 
which is essentially small, 
the set of its isomorphism classes is 
denoted by $\Iso(\cal{A})$.

For a complex 
$\bu{M}=(\cdots \to M_{i}\xrightarrow{d_i} M_{i+1} \to \cdots )$ 
in an abelian category $\cal{A}$,
its homology is denoted by $H_*(\bu{M})$. 

For a set $S$, 
we denote by $\bigl| S \bigr|$ its cardinality.

%%%%%%%%%%%%%%%%%%%%%%%%%%%%%%%%%%%%%%%%%%%%%%%%%%
%%%%%%%%%%%%%%%%%%%%%%%%%%%%%%%%%%%%%%%%%%%%%%%%%%
%%%%%%%%%%%%%%%%%%%%%%%%%%%%%%%%%%%%%%%%%%%%%%%%%%

\section{Hall algebras of complexes}
\label{sec:sec1}

We summarize necessary definitions and properties of  
Hall algebras of $\bb{Z}/2\bb{Z}$-graded complexes.
Most of the materials were introduced and shown in \cite{B}.

%%%%%%%%%%%%%%%%%%%%%%%%%%%%%%%%%%%%%%%%%%%%%%%%%%
%%%%%%%%%%%%%%%%%%%%%%%%%%%%%%%%%%%%%%%%%%%%%%%%%%

\subsection{$\bb{Z}/2\bb{Z}$-periodic complexes}

We will recall the basic definitions in \cite[\S3]{B}.
Let us fix an abelian category $\cal{A}$.

Let $\cal{C}(\cal{A}) \equiv \cal{C}_{\bb{Z}/2\bb{Z}}(\cal{A})$ 
be the abelian category 
of $\bb{Z}/2\bb{Z}$-graded complexes in $\cal{A}$.  
An object $\bu{M}$ of $\cal{C}(\cal{A})$ consists of 
the following diagram in $\cal{A}$:
\begin{align*}
\cpx{M_1}{d^M_1}{d^M_0}{M_0}, \quad d^M_{i+1}\circ d^M_i=0.
\end{align*} 
A morphism $\bu{s}: \bu{M} \to \bu{N}$ 
consists of a diagram
\begin{align*}
\xymatrix@R=3em@C=3em{
  M_1 \ar[d]_{s_1} \ar@<.5ex>[r]^{d^M_1} 
& M_0 \ar[d]^{s_0} \ar@<.5ex>[l]^{d^M_0}
\\
  N_1 \ar@<.5ex>[r]^{d^N_1}
& N_0 \ar@<.5ex>[l]^{d^N_0}
}
\end{align*}
with $s_{i+1} \circ d^M_{i} = d^N_i\circ s_{i}$.
Here and hereafter indices in the diagram of an object 
in $\cal{C}(\cal{A})$ 
are understood by modulo $2$.
We also denote by $0 \equiv \bu{0}$ the trivial 
$\bb{Z}/2\bb{Z}$-graded complex
(where the graded parts are equal to the zero object of $\cal{A}$).

Two morphisms $\bu{s}, \bu{t}: \bu{M} \to \bu{N}$ 
are said to be homotopic 
if there are morphisms $h_i: M_i\to N_{i+1}$ such that
$t_i-s_i=d'_{i+1} \circ h_{i}+h_{i+1}\circ d_{i}$.
Denote by $\Ho(\cal{A}) \equiv \Ho_{\bb{Z}/2\bb{Z}}(\cal{A})$ 
the category obtained from $\cal{C}(\cal{A})$ 
by identifying homotopic morphisms.

Let us also denote by 
$\cal{C}(\cal{P}) \equiv \cal{C}_{\bb{Z}/2\bb{Z}}(\cal{P}) 
 \subset \cal{C}(\cal{A})$
the full subcategory 
whose objects are complexes of projectives in $\cal{A}$.

For an object $\bu{M}$ of $\cal{C}(\cal{A})$,
we define its class $\cl{\bu{M}}$ 
in the Grothendieck group $K(\cal{A})$ to be 
$$
 \cl{\bu{M}} := \cl{M_0} - \cl{M_1} \in K(\cal{A}).
$$
%\\

The shift functor $[1]$ of complexes induces an involution 
$\cal{C}(\cal{A}) \stackrel{*}{\longleftrightarrow} 
\cal{C}(\cal{A})$.
This involution shifts the grading and changes the sign of the differential 
as follows:
\begin{align*}
\bu{M} = \cpx{M_1}{d^M_1}{d^M_0}{M_0}
\quad \stackrel{*}{\longleftrightarrow} \quad 
\bu{M}^{*} = \cpx{M_0}{-d^M_0}{-d^M_1}{M_1}
\end{align*}

%%%%%%%%%%%%%%%%%%%%%%%%%%%%%%%%%%%%%%%%%%%%%%%%%%
%%%%%%%%%%%%%%%%%%%%%%%%%%%%%%%%%%%%%%%%%%%%%%%%%%

\subsection{Bridgeland's Hall algebra of complexes}

Let us recall the definition of 
the ordinary Hall algebra.
For an abelian category $\cal{A}$ satisfying the condition (a),
consider the vector space
$$
 \cal{H}(\cal{A}) := 
 \bigoplus_{A \in \Iso(\cal{A})} \bb{C}[A]
$$
linearly spanned by symbols $[A]$ with $A$ running through 
the set $\Iso(\cal{A})$  
of isomorphism classes of objects in $\cal{A}$.
Then, by Ringel \cite{R1},
the following operation $\diamond$ defines on $\cal{H}(\cal{A})$ 
a structure of unital associative algebra over $\bb{C}$:
\begin{align*}
[A] \diamond [B] := 
\sum_{C \in \Iso(\cal{A})} 
\dfrac{ \bigl| \Ext^1_{\cal{A}}(A,B)_{C} \bigr| }
      { \bigl| \Hom_{\cal{A}}(A,B) \bigr|} 
[C].
\end{align*}
Here 
\begin{align*}
\Ext^1_{\cal{A}}(A,B)_C \subset \Ext^1_{\cal{A}}(A,B)
\end{align*}
is the set parametrizing extensions of $B$ by $A$
with the middle term isomorphic to $C$.
The class $[0]$ of the zero object is the unit for this product $\diamond$,
and the algebra $(\cal{H}(\cal{A}), \diamond, [0])$ is called 
the Hall algebra of $\cal{A}$. 
\\

For an abelian category $\cal{A}$ 
satisfying the condition (a),
Let $\cal{H}(\cal{C}(\cal{P}))$
be the Hall algebra 
of the category $\cal{C}(\cal{P})$.
As a $\bb{C}$-vector space we have  
$\cal{H}(\cal{C}(\cal{P})) = 
 \bigoplus_{\bu{M} \in \Iso(\cal{C}(\cal{P}))} \bb{C}[\bu{M}]$,
and the product is given by the formula
\begin{align*}
[\bu{M}] \diamond [\bu{N}] := 
\sum_{\bu{L} \in \Iso(\cal{C}(\cal{P}))} 
\dfrac{\bigl| \Ext^1_{\cal{C}(\cal{A})}(\bu{M},\bu{N})_{\bu{L}} \bigr|}
      {\bigl| \Hom_{\cal{C}(\cal{A})}(\bu{M},\bu{N}) \bigr|} 
[\bu{L}].
\end{align*}
Since  the subcategory $\cal{C}(\cal{P})\subset\cal{C}(\cal{A})$ 
is closed under extensions 
and 
since the space $\Ext^1_{\cal{C}(\cal{A})}(\bu{M},\bu{N})$ 
is always of finite dimension,
this expression makes sense. 
Indeed,  according to \cite[Lemma~3.3]{B} we have 
$\Ext^1_{\cal{C}(\cal{A})}(\bu{M},\bu{N})
 \cong \Hom_{\Ho(\cal{A})}(\bu{M},\bu{N}^*)$.

\begin{rmk}
\label{rmk:normalization}
As in \cite[\S2.3]{B}, using
$[[\bu{M}]] := [\bu{M}]/\bigl| \Aut_{\cal{C}(\cal{A})}(\bu{M}) \bigr|$
one can rewrite the multiplication as
$$
 [[\bu{M}]] \diamond [[\bu{N}]]
 = \sum_{\bu{L}} g^{\bu{L}}_{\bu{M},\bu{N}} [[\bu{L}]]
$$
with 
\begin{align}
\label{eq:g}
 g^{\bu{L}}_{\bu{M},\bu{N}} := 
 \bigl| \left\{\bu{N}' \subset \bu{L} 
    \mid \bu{N}' \cong \bu{N},\ \bu{L}/\bu{N} \cong \bu{M} 
 \right\}\bigr|
\end{align}
for $\bu{L},\bu{M},\bu{N} \in \Obj(\cal{C}(\cal{A}))$.
Here we used the well-known formula
\begin{align}
\label{eq:gehaaa}
 g^{A}_{B,C}
=\dfrac{\bigl| \Ext_{\cal{A}}(B,C)_A \bigr|}
       {\bigl| \Hom_{\cal{A}}(B,C)   \bigr|}
 \dfrac{\bigl| \Aut_{\cal{A}}(A) \bigr|}
       {\bigl| \Aut_{\cal{A}}(B) \bigr|
        \bigl| \Aut_{\cal{A}}(C) \bigr|}
\end{align}
for any abelian category $\cal{A}$ 
and its objects $A,B,C$.
Then the associativity of $\diamond$,
that is, 
$([\bu{M}] \diamond [\bu{N}]) \diamond [\bu{P}]
= [\bu{M}] \diamond ([\bu{N}] \diamond [\bu{P}])$
for any $\bu{M},\bu{N},\bu{P} \in \Obj(\cal{C}(\cal{A}))$,
is equivalent to the formula
\begin{equation}
\label{eq:dh:assoc}
 \sum_{\bu{Q}} g_{\bu{M},\bu{N}}^{\bu{Q}} g_{\bu{Q},\bu{P}}^{\bu{R}}
=\sum_{\bu{S}} g_{\bu{M},\bu{S}}^{\bu{R}} g_{\bu{N},\bu{P}}^{\bu{S}}
\end{equation}
for any tuple $(\bu{M},\bu{N},\bu{P},\bu{Q})$ 
of objects of $\cal{C}(\cal{A})$.

We will sometimes use this renormalized generators $[[\bu{M}]]$ 
in the argument.
\end{rmk}

Hereafter we consider an abelian category $\cal{A}$ 
satisfying the conditions (a) -- (c).
Define 
$\cal{H}_{\tw}(\cal{C}(\cal{P}))=(\cal{H}(\cal{C}(\cal{P})),*,[0])$ 
to be the pair of the vector space $\cal{H}(\cal{C}(\cal{P}))$ 
with the twisted multiplication
\begin{equation}
\label{eq:tw-mul}
[\bu{M}]*[\bu{N}]
:=t^{\<M_0,N_0\>+\<M_1,N_1\>} 
  [\bu{M}] \diamond [\bu{N}]
\end{equation}
and the class of the zero object.
Here we used the Euler form  
$\<A,B\> 
 := \sum_{i \in \bb{Z}} (-1)^i \dim_{\frk{k}}\Ext^i_{\cal{A}}(A,B)$
on $\cal{A}$.
%Note that the sum is finite by our assumptions on $\cal{A}$.
As is well known, this form descends to the one 
on the Grothendieck group $K(\cal{A})$ of $\cal{A}$,
which is denoted by the same symbol.
We will also use the symbol
\begin{align}
\label{eq:euler-dash}
\<\bu{M},\bu{N}\>' := \<M_0,N_0\>+\<M_1,N_1\>.
\end{align}

Now we can introduce Bridgeland's Hall algebra:

\begin{df} 
Let $\cal{DH}(\cal{A})$  be
the localization of the algebra $\cal{H}_{\tw}(\cal{C}(\cal{P}))$ 
with respect to the elements $[\bu{M}]$ 
corresponding to acyclic complexes $\bu{M}$: 
\begin{align*}
 \cal{DH}(\cal{A})
 := \cal{H}_{\tw}(\cal{C}(\cal{P}))
    \bigl[ [\bu{M}]^{-1} \mid H_*(\bu{M})=0\bigr].
\end{align*}
Then the tuple $(\cal{DH}(\cal{A}), [0], *)$ 
is a unital associative algebra.
\end{df}

\begin{rmk}
\label{rmk:localize}
Let us recall this localization process in detail.
For $P\in \Obj(\cal{P})$, we define 
\begin{align*}
\bu{{K_P}}    := \cpx{P}{\Id}{0}{P}, \qquad 
\bu{{K_P}}^*  := \cpx{P}{0}{-\Id}{P},
\end{align*}
which are obviously acyclic.
By \cite[Lemma~3.2]{B},
every acyclic complex of projectives 
$\bu{M} \in \Obj(\cal{C}(\cal{P}))$
can be written as 
\begin{align}
\label{eq:acyclic}
 \bu{M} \cong \bu{{K_P}} \oplus \bu{{K_Q}}^*,
\end{align}
and $P,Q \in \Obj(\cal{P})$ are determined unique up to isomorphism. 
The complexes $\bu{{K_P}}$ and $\bu{{K_P}}^*$ enjoy the relations
\begin{align*}
&[\bu{{K_P}}] * [\bu{M}] 
= t^{\<\cl{P},\bu{\cl{M}}\>} [\bu{{K_P}} \oplus \bu{M}]
= t^{ (\cl{P},\bu{\cl{M}}) } \bu{M} * \bu{{K_P}},
\\
&[\bu{{K_P}}^*] * [\bu{M}] 
= t^{-\<\cl{P},\bu{\cl{M}}\>} [\bu{{K_P}}^* \oplus \bu{M}]
= t^{- (\cl{P},\bu{\cl{M}}) } \bu{M} * \bu{{K_P}}^*.
\end{align*}
Here we used the symmetrized Euler form:
\begin{align}
\label{eq:sym-euler}
 (\alpha,\beta) := \<\alpha,\beta\> + \<\beta,\alpha\>.
\end{align}
Thus the subset $\left\{ [\bu{M}] \mid H_*(\bu{M})=0\right\}$ 
of $\Iso(\cal{C}(\cal{P}))$ satisfies the Ore condition,
and one can consider the localization 
of the non-commutative algebra $\cal{H}_{\tw}(\cal{C}(\cal{P}))$ 
with respect to this subset.
$\cal{DH}(\cal{A})$ is the algebra obtained by this localization process.
As explained in \cite[\S3.6]{B},
this is the same as localizing by the elements 
$[\bu{{K_P}}]$ and $[\bu{{K_P}^*}]$ 
for all objects $P\in \Obj(\cal{P})$. 
\end{rmk}

For an element $\alpha \in K(\cal{A})$, we define 
\begin{align*}
K_\alpha := [\bu{{K_P}}] * [\bu{{K_Q}}]^{-1},\quad
K_\alpha^* := [\bu{{K_P}}^*] * [\bu{{K_Q}}^*]^{-1},
\end{align*}
where we expressed $\alpha = \cl{P}-\cl{Q}$ 
using the classes of some projectives $P,Q \in \Obj(\cal{P})$.
These are well defined by \cite[Lemmas~3.4, 3.5]{B},
and also we have 
\begin{align}
\label{eq:KM}
&K_{\alpha} * [\bu{M}] 
%= t^{\<\alpha,\bu{\cl{M}}\>} [K_{\alpha} \oplus \bu{M}]
 = t^{(\alpha,\bu{\cl{M}})}   [\bu{M}] * K_{\alpha},
 \qquad
%\\
%\nonumber
%&
 K^*_{\alpha} * [\bu{M}] 
%= t^{-\<\alpha,\bu{\cl{M}}\>} [K^*_{\alpha} \oplus \bu{M}]
 = t^{-(\alpha,\bu{\cl{M}})} [\bu{M}] * K_{\alpha},
\\
\label{eq:KK}
&K_{\alpha} * K_{\beta} = K_{\alpha+\beta}, \qquad
 K^*_{\alpha} * K^*_{\beta} = K^*_{\alpha+\beta}, \qquad
\\
\nonumber
&[K_\alpha,K_\beta] = [K_\alpha,K_\beta^*] = [K_\alpha^*,K_\beta^*]=0
\end{align}
in the algebra $\cal{DH}(\cal{A})$
for arbitrary $\alpha, \beta \in K(\cal{A})$
and $\bu{M} \in \Obj(\cal{C}(\cal{P}))$.

\section{Coproduct}

%%%%%%%%%%%%%%%%%%%%%%%%%%%%%%%%%%%%%%%%%%%%%%%%%%
%%%%%%%%%%%%%%%%%%%%%%%%%%%%%%%%%%%%%%%%%%%%%%%%%%

\subsection{Green's coproduct}
\label{ssubsec:cpr}

Let us recall the coalgebra structure 
on the ordinary Hall algebra 
introduced by Green \cite{Gr}.
Here one should consider a completion of the algebra.
We recommend \cite[Lecture 1]{S} for a nice review 
on this topic.
\\

Assume that the abelian category $\cal{A}$ satisfies 
the conditions (a), (b), (c) and (e).
Then the Hall algebra 
$\cal{H}(\cal{A}) = 
 \bigoplus_{A \in \Iso(\cal{A})} \bb{C}[A]$
is naturally graded 
by the Grothendieck group $K(\cal{A})$ of $\cal{A}$:
\begin{align*}
 \cal{H}(\cal{A}) 
 = \bigoplus_{\alpha \in K(\cal{A})}
 \cal{H}(\cal{A})[\alpha],
 \quad
 \cal{H}(\cal{A})[\alpha]:=\bigoplus_{\cl{A}=\alpha}\bb{C} [A].
\end{align*}
For $\alpha,\beta \in K(\cal{A})$, 
set
\begin{align*}
\cal{H}(\cal{A})[\alpha] 
 \mathop{\wh{\otimes}}_{\bb{C}} 
\cal{H}(\cal{A})[\beta]
 :=\prod_{\cl{A}=\alpha, \, \cl{B}=\beta} 
    \bb{C} [A] \mathop{\otimes}_{\bb{C}} \bb{C} [B],
 \qquad
\cal{H}(\cal{A}) 
 \mathop{\wh{\otimes}}_{\bb{C}} 
\cal{H}(\cal{A})
 :=\prod_{\alpha,\beta \in K(\cal{A})}
   \cal{H}(\cal{A})[\alpha] 
    \mathop{\wh{\otimes}}_{\bb{C}}
   \cal{H}(\cal{A})[\beta].
\end{align*}

Hereafter we will suppress the symbol $\bb{C}$ 
at the tensor product symbol $\otimes$.
The space $\cal{H}(\cal{A}) \ctp \cal{H}(\cal{A})$
consists of all formal linear combinations
$\sum_{A,B} c_{A,B} [A] \otimes [B]$.
The coassociativity in this completed space 
will be called the topological coassociativity.

\begin{fct}[{Green \cite{Gr}}]
\label{fct:Gr}
Assume that  $\cal{A}$ satisfies 
the conditions (a), (b), (c) and (e).
\begin{enumerate}
\item[(1)]
Then the maps
\begin{align*}
&\Delta'_{\cal{H}(\cal{A})} \equiv \Delta': 
 \cal{H}(\cal{A}) \longrightarrow 
 \cal{H}(\cal{A}) \ctp \cal{H}(\cal{A}), 
\\
&\Delta'([A]) :=
  \sum_{B,C} t^{\<B,C\>} 
  \dfrac{\bigl| \Ext^1_{\cal{A}}(B,C)_A \bigr|}
        {\bigl| \Hom_{  \cal{A}}(B,C)   \bigr|}
  \dfrac{\bigl| \Aut_{  \cal{A}}(A)     \bigr|}
        {\bigl| \Aut_{  \cal{A}}(B)     \bigr| 
         \bigl| \Aut_{  \cal{A}}(C)     \bigr|}
   [B] \otimes [C],
\end{align*}
and
\begin{align*}
\ep:\cal{H}(\cal{A}) \longrightarrow \bb{C},\qquad
\ep([A]):=\delta_{A,0}.
\end{align*} 
define a topological counital coassociative coalgebra structure
on $\cal{H}(\cal{A})$.

\item[(2)]
Assume moreover that $\cal{A}$ also satisfies 
the condition (d).
Then the tuple
$(\cal{H}(\cal{A}),\diamond,[0],\Delta',\ep)$,
is a topological bialgebra defined over $\bb{C}$.
That is, the map 
$\Delta': \cal{H}(\cal{A}) \to 
 \cal{H}(\cal{A}) \ctp \cal{H}(\cal{A})$
is a  homomorphism of $\bb{C}$-algebras.
\end{enumerate}
\end{fct}

\begin{rmk}
As in Remark \ref{rmk:normalization},
we can rewrite the definition of our coproduct 
using $[[A]] := [A]/\bigl| \Aut_{\cal{A}}(A) \bigr|$ 
into the form 
\begin{align}
\label{eq:Delta':ord:norm}
 \Delta'([[A]])
 = \sum_{B,C}   t^{\<B,C\>}
   g^A_{B,C} \dfrac{ a_B a_C }{ a_A }
   [[B]] \otimes [[C]]
\end{align}
with $a_A := \bigl| \Aut_{\cal{A}}(A) \bigr|$ and 
$g^{A}_{B,C} := 
 \bigl| \{C' \subset A  \mid C' \cong C,\ A/C' \cong B \} \bigr|
$.
\end{rmk}

%%%%%%%%%%%%%%%%%%%%%%%%%%%%%%%%%%%%%%%%%%%%%%%%%%%
%%%%%%%%%%%%%%%%%%%%%%%%%%%%%%%%%%%%%%%%%%%%%%%%%%%

\subsection{Naive coproduct on Hall algebra of complexes}
\label{subsec:naive}

To introduce a coproduct on Bridgeland's Hall algebra
$\cal{DH}(\cal{A})$,
we begin with the unlocalized algebra $\cal{H}_{\tw}(\cal{C}(\cal{P}))$.

\begin{dfn}
For an abelian category $\cal{A}$ satisfying
the conditions (a), (b), (c) and (e),
we define a $\bb{C}$-linear map 
$$
\Delta'_{\chi}: 
 \cal{H}(\cal{C}(\cal{P}))
 \longrightarrow
 \cal{H}(\cal{C}(\cal{P})) 
 \ctp
 \cal{H}(\cal{C}(\cal{P}))
$$
by 
\begin{align}
\label{eq:delta':hcp}
 \Delta'_{\chi}([\bu{L}])
 := \sum_{\bu{M},\bu{N}} 
     t^{\chi(\bu{M},\bu{N})} 
     \dfrac{\bigl| \Ext^1_{\cal{C}(\cal{A})}(\bu{M},\bu{N})_{\bu{L}} \bigr|}
           {\bigl| \Hom_{  \cal{C}(\cal{A})}(\bu{M},\bu{N}) \bigr|}
     \dfrac{\bigl| \Aut_{\cal{C}(\cal{A})}(\bu{L}) \bigr|}
           {\bigl| \Aut_{\cal{C}(\cal{A})}(\bu{M}) \bigr|
            \bigl| \Aut_{\cal{C}(\cal{A})}(\bu{N}) \bigr|}
     [\bu{M}] \otimes [\bu{N}].
\end{align}
Here $\chi$ is an arbitrary map from
$\Iso(\cal{C}(\cal{P})) \times \Iso(\cal{C}(\cal{P}))$ 
to $\bb{Z}$ % (or $\bb{C}$)
satisfying the condition
\begin{align}
\label{eq:chi}
\begin{split}
 \text{If} \quad 
&\cl{M_i}=\cl{P_i}+\cl{Q_i}
 \ \text{ and } \ 
 \cl{R_i}=\cl{Q_i}+\cl{N_i}
 \ \text{ for }\ 
 i=0,1,
\\
 \text{then} \quad
&\chi(\bu{M},\bu{N}) + \chi(\bu{P}, \bu{Q}) 
=\chi(\bu{P},\bu{R}) + \chi(\bu{Q}, \bu{N}).
\end{split}
\end{align}
We also define 
$$
  \ep: \cal{H}(\cal{C}(\cal{P})) \longrightarrow \bb{C}
$$
by
\begin{align}
\label{eq:counit:tw}
 \ep([\bu{M}]) := \delta_{\bu{M},0}.
\end{align}
\end{dfn}

\begin{rmk}
\label{rmk:coalg}
\begin{enumerate}
\item
Examples for $\chi$ is the Euler form $\<\cdot,\cdot\>'$ 
appearing in \eqref{eq:euler-dash}.
We can transpose the entries and can multiply the form by scalar.
So $\chi(\bu{M},\bu{N}):=-\<\bu{N},\bu{M}\>'$ also satisfies
the condition \eqref{eq:chi}.

\item
As in Remark \ref{rmk:normalization},
we can rewrite our coproduct using 
$[[\bu{L}]] := [\bu{L}]/\bigl| \Aut_{\cal{C}(\cal{A})}(\bu{L}) \bigr|$ 
into the form 
$$
 \Delta'_{\chi}([[\bu{L}]])
 = \sum_{\bu{M},\bu{N}}  
    t^{\chi(\bu{M},\bu{N})}
    g^{\bu{L}}_{\bu{M},\bu{N}} 
    \dfrac{ a_{\bu{M}} a_{\bu{N}} }{ a_{\bu{L}} }
    [[\bu{M}]] \otimes [[\bu{N}]] 
$$
with 
\begin{align}
\label{eq:a}
 a_{\bu{L}} := \bigl| \Aut_{\cal{C}(\cal{A})}(\bu{L}) \bigr|.
\end{align}
We will often use the symbols $[[\bu{L}]]$ and $a_{\bu{L}}$ 
in the following argument.
\end{enumerate}
\end{rmk}

\begin{lem}
\label{lem:coassoc:hcp}
Assume that an abelian category $\cal{A}$ satisfying
the conditions (a), (b), (c) and (e).
Then the triple 
$\bigl( \cal{H}(\cal{C}(\cal{P})),\Delta'_{\chi},\ep \bigr)$ 
is a topological counital coassociative coalgebra.
\end{lem}

\begin{proof}
Our argument is almost the same as 
that for the proof of the coassociativity of 
the ordinary Hall algebra (Fact~\ref{fct:Gr}(1)),
but for completeness we will write down a proof.

The topological coassociativity of $\Delta'_{\chi}$,
that is 
$(\Delta'_{\chi} \otimes 1)\circ\Delta'_{\chi}([[\bu{L}]])
=(1 \otimes \Delta'_{\chi})\circ\Delta'_{\chi}([[\bu{L}]])$
on the completed tensor product 
$\cal{H}(\cal{C}(\cal{P})) 
 \ctp \cal{H}(\cal{C}(\cal{P})) 
 \ctp \cal{H}(\cal{C}(\cal{P}))$
for any object $\bu{L}$ in $\cal{H}(\cal{C}(\cal{P}))$,
is equivalent to the following formula
\begin{align*}
&\sum_{\bu{M},\bu{N},\bu{P},\bu{Q}}
  t^{ \chi(\bu{M},\bu{N}) }
  g^{\bu{L}}_{\bu{M},\bu{N}}
  \dfrac{a_{\bu{M}} a_{\bu{N}}}{a_{\bu{L}}}
  t^{ \chi(\bu{P},\bu{Q}) }
  g^{\bu{M}}_{\bu{P},\bu{Q}}
  \dfrac{a_{\bu{P}} a_{\bu{Q}}}{a_{\bu{M}}}
  [[\bu{P}]] \otimes [[\bu{Q}]] \otimes [[\bu{N}]] \\
=
&\sum_{\bu{M},\bu{N},\bu{P},\bu{Q}}
  t^{ \chi(\bu{M},\bu{N}) }
  g^{\bu{L}}_{\bu{M},\bu{N}}
  \dfrac{a_{\bu{M}} a_{\bu{N}}}{a_{\bu{L}}}
  t^{ \chi(\bu{P},\bu{Q}) }
  g^{\bu{N}}_{\bu{P},\bu{Q}}
  \dfrac{a_{\bu{P}} a_{\bu{Q}}}{a_{\bu{N}}}
  [[\bu{M}]] \otimes [[\bu{P}]] \otimes [[\bu{Q}]].
\end{align*}
One can see that it is equivalent to the formula
\begin{align*}
\sum_{\bu{M}}
  t^{ \chi(\bu{M},\bu{N}) + \chi(\bu{P},\bu{Q})}
  g^{\bu{L}}_{\bu{M},\bu{N}} g^{\bu{M}}_{\bu{P},\bu{Q}}
  \dfrac{ a_{\bu{N}} a_{\bu{P}} a_{\bu{Q}} }{ a_{\bu{L}} }
=
\sum_{\bu{R}}
  t^{ \chi(\bu{P},\bu{R}) + \chi(\bu{Q},\bu{N}) }
  g^{\bu{L}}_{\bu{P},\bu{R}}   g^{\bu{R}}_{\bu{Q},\bu{N}}
  \dfrac{ a_{\bu{N}} a_{\bu{P}} a_{\bu{Q}} }{ a_{\bu{L}} }
\end{align*}
for any tuple $(\bu{L},\bu{N},\bu{P},\bu{Q})$ of 
objects in $\cal{C}(\cal{P})$.
By the condition \eqref{eq:chi},
the last formula is reduced to
$$
\sum_{\bu{M}}
  g^{\bu{M}}_{\bu{P},\bu{Q}} g^{\bu{L}}_{\bu{M},\bu{N}} 
=
\sum_{\bu{R}}
  g^{\bu{L}}_{\bu{P},\bu{R}}   g^{\bu{R}}_{\bu{Q},\bu{N}},
$$
which is nothing but the consequence of the 
associativity \eqref{eq:dh:assoc} of the product $\diamond$.

It is easy to see that $\ep$ gives a counit.
Thus we have the conclusion.
\end{proof}

\begin{rmk}
\label{rmk:failure}
Now we want to consider the localized algebra 
$
\cal{DH}(\cal{A})
 = \cal{H}_{\tw}(\cal{C}(\cal{P}))
    \bigl[ [\bu{M}]^{-1} \mid H_*(\bu{M})=0 \bigr]$.
Let us denote by 
$$
 S := \left\{ [\bu{M}] \mid H_*(\bu{M})=0 \right\}
$$ 
the subset of $C := \cal{H}_{\tw}(\cal{C}(\cal{P}))$
used in the localization.
If $S$ spans a coideal 
with respect to the coproduct $\Delta'_\chi$,
in other words 
$\Delta'_\chi(S) \subset S \ctp C + C \ctp S$,
then
the coalgebra structure 
$\bigl( \cal{H}(\cal{C}(\cal{P})),
        \Delta'_{\chi} \bigr)$
descends to 
$\bigl( C/\mathop{\text{Span}}(S), \Delta'_\chi \bigr)$.

However, this strategy does not work.
Consider the exact sequence
\begin{align*}
\xymatrix@1{
  0     \ar[r]
&\bu{N} \ar[r]
&\bu{L} \ar[r]
&\bu{M} \ar[r]
&0
}
\end{align*}
in $\cal{C}(\cal{P})$,
which is expressed as the exact commutative diagram
\begin{align}
\label{diag:failure}
\xymatrix{
 0 \ar[r] 
&0 \ar[r]         \ar@<-.4ex>[d]
&P \ar[r]^{\Id_P} \ar@<-.4ex>[d]_{\Id_P}
&P \ar[r]         \ar@<-.4ex>[d]
&0
\\
 0 \ar[r]  
&P \ar[r]_{\Id_P} \ar@<-.4ex>[u]
&P \ar[r]         \ar@<-.4ex>[u]_{0}
&0 \ar[r]         \ar@<-.4ex>[u]
&0
}
\end{align}
with $P \in \Obj(\cal{P})$.
In this case we have
$H_*(\bu{L})=0$ but $H_*(\bu{M})\neq0$ and 
$H_*(\bu{N})\neq0$.
Thus we have 
$\Delta'_\chi(S) \not\subset S \ctp C + C \ctp S$.

In the next subsection 
we impose certain conditions on the exact sequences 
in $\cal{C}(\cal{P})$ which are counted in the desired coproduct formula.
For such a purpose,
it is convenient to recall the notion of exact categories 
in the sense of Quillen \cite{Q}.
\end{rmk}

%%%%%%%%%%%%%%%%%%%%%%%%%%%%%%%%%%%%%%%%%%%%%%%%%%%
%%%%%%%%%%%%%%%%%%%%%%%%%%%%%%%%%%%%%%%%%%%%%%%%%%%

\subsection{Hall algebras of exact categories and descent of coproduct}
\label{subsec:exact}

We follow the description of an exact category
given in \cite[\S2]{H} and \cite[Appendix~A]{K}.

An exact category in the sense of Quillen \cite{Q}
is a pair $(\cal{A},\cal{E})$ of 
an additive category $\cal{A}$ 
and a class $\cal{E}$ of kernel-cokernel pairs $(f,g)$ 
closed under isomorphisms,
satisfying the following axioms.
A deflection mentioned in these axioms 
is the first component of some $(f,g) \in \cal{E}$,
and a inflation is the second component.
A pair $(f,g) \in \cal{E}$ will be called a conflation.
\begin{enumerate}
\item[(Ex0)] 
 $0 \xrightarrow{\ \Id \ } 0$ is a deflation.
\item[(Ex1)]
 The composition of two deflations is a deflation.
\item[(Ex2)]
 For any $f: Z' \to Z$ and a deflation $d: Y \to Z$,
 there is a cartesian square
 $$
 \xymatrix{
  Y' \ar[r]^{d'} \ar[d]_{f'} 
     \ar@{}[rd]|{\square}
 &Z' \ar[d]^{f}
 \\
  Y  \ar[r]_{d} 
 &Z 
 }
 $$
 with $d'$ a deflation.
\item[{(Ex$2^{op}$)}]
 For any $f: X \to X'$ and each inflation $i: X \to Y$,
 there is a cocartesian square
 $$
 \xymatrix{
  X \ar[r]^{i} \ar[d]_{f'} 
     \ar@{}[rd]|{\square}
 &Y \ar[d]^{f}
 \\
  X'  \ar[r]_{i'} 
 &Y' 
 }
 $$
 with $i'$ an inflation.
\end{enumerate}

The notion of Hall algebra for an exact category 
was introduced in \cite{H},
and the associativity of the algebra was shown:

\begin{fct}[{\cite[Theorem.~3]{H}}]
\label{fct:exact}
Let $(\cal{A},\cal{E})$ be an exact category 
with $\cal{A}$ essentially small and 
having finite morphism spaces.

Let us define a free $\bb{Z}$-module 
$$
 \cal{H}(\cal{A},\cal{E})
 := \bigoplus_{X \in \Iso(\cal{A})} \bb{Z} [X]
$$ 
and introduce a binary operator $\diamond$ by 
$$
 [X] \diamond [Y]
 := \sum_{Z} 
    \dfrac{\bigl| W^Z_{X,Y} \bigr|}
          {\bigl| \Hom_{\cal{A}}(X,Y) \bigr|} 
    [Z].
$$
Here $W^Z_{X,Y}$ is the set of all conflations 
of the form $Y \to Z \to X$.

Then $(\cal{H}(\cal{A},\cal{E}), \diamond, [0])$ 
is a unital associative ring.
\end{fct}

Of course the Hall algebra $\cal{H}(\cal{A})$ 
for an abelian category $\cal{A}$ satisfying the conditions (a)
coincides with $\cal{H}(\cal{A},\cal{E})$ 
(after tensoring $\bb{C}$),
where $\cal{E}$ is the set of all exact sequences 
in the abelian category $\cal{A}$.
\\

Now we want to introduce a special subset $\cal{E}_0$ 
of the set of all exact sequences in $\cal{C}(\cal{P})$,
and to define the Hall algebra associated to 
exact category $\bigl( \cal{C}(\cal{P}), \cal{E} \bigr)$ 

\begin{dfn}
Let $\cal{E}_0$ be the class of exact sequences 
\begin{align*}
\xymatrix@1{
  0     \ar[r]
&\bu{N} \ar[r]
&\bu{L} \ar[r]
&\bu{M} \ar[r]
&0
}
\end{align*}
in $\cal{C}(\cal{P})$ satisfying the condition
\begin{align}
\label{eq:E_0:cond}
H_i(M) \neq 0 \Longrightarrow H_{i+1}(N) = 0 
\ \text{ for both } \ i=0,1.
\end{align}
\end{dfn}

The motivation of this definition comes from 
Remark~\ref{rmk:failure}.
The exact sequence \eqref{diag:failure} 
is excluded from $\cal{E}_0$ 
by the condition \eqref{eq:E_0:cond}.

\begin{prop}
The pair $\bigl( \cal{C}(\cal{P}), \cal{E}_0 \bigr)$ 
is an exact category.
\end{prop}

\begin{proof}
The axioms (Ex0) and (Ex1) are trivially true.
$\cal{C}(\cal{P})$ is closed under cartesian and cocartesian squares,
so it is enough to check that
for a diagram 
$$
\xymatrix{
 0       \ar[r]
&\bu{N'} \ar[r]^{i'} \ar[d]%_{i'}
         \ar@{}[rd]|{\square}
&\bu{L'} \ar[r]^{d'} \ar[d]%_{f'} 
         \ar@{}[rd]|{\square}
&\bu{M'} \ar[r] \ar[d] 
&0
\\
 0       \ar[r]
&\bu{N}  \ar[r]_{i} 
&\bu{L}  \ar[r]_{d} 
&\bu{M}  \ar[r]
&0
}
$$
if the lower row belongs to $\cal{E}_0$ 
then the upper row belongs to $\cal{E}_0$
and vice-a-versa.
But it can be checked by simple diagram chasing.
\end{proof}

Now we introduce 

\begin{dfn}
For an abelian category $\cal{A}$ 
satisfying the conditions (a) -- (c),
denote by 
$\cal{H}_{\tw}(\cal{C}(\cal{P},\cal{E}_0))
 =(\cal{H}(\cal{C}(\cal{P})),*_{\cal{E}_0})$ 
the $\bb{C}$-vector space 
$\cal{H}(\cal{C}(\cal{P})) 
 = \bigoplus_{[\bu{M}] \in \Iso(\cal{C}(\cal{P}))} \bb{C} [\bu{M}]$ 
with the multiplication 
$$
[\bu{M}] *_{\cal{E}_0} [\bu{N}] := 
t^{\<\bu{M},\bu{N}\>'}
\sum_{\bu{L} \in \Iso(\cal{C}(\cal{P}))} 
\dfrac{\bigl| W^{\bu{L}}_{\bu{M},\bu{N}} \bigr|}
      {\bigl| \Hom_{\cal{C}(\cal{A})}(\bu{M},\bu{N}) \bigr|} 
[\bu{L}].
$$
Here $W^{\bu{L}}_{\bu{M},\bu{N}}$ denotes 
the set of all conflations $ \bu{N} \to \bu{L} \to \bu{M}$ 
in $\cal{E}_0$.
\end{dfn}

By Fact~\ref{fct:exact},
$\cal{H}_{\tw}(\cal{C}(\cal{P},\cal{E}_0))$
is a unital associative algebra.
\\

Using the exact category $\bigl(\cal{C}(\cal{P}), \cal{E}_0 \bigr)$,
we introduce a coproduct on the Hall algera of complexes.

\begin{dfn}
For an abelian category $\cal{A}$ satisfying
the conditions (a), (b), (c) and (e),
we define a $\bb{C}$-linear map 
$$
\Delta'_{\chi,\cal{E}_0}: 
 \cal{H}(\cal{C}(\cal{P}))
 \longrightarrow
 \cal{H}(\cal{C}(\cal{P})) 
 \ctp
 \cal{H}(\cal{C}(\cal{P}))
$$
by 
\begin{align*}
 \Delta'_{\chi,\cal{E}_0}([\bu{L}])
 := \sum_{\bu{M},\bu{N}} 
     t^{\chi(\bu{M},\bu{N})} 
     \dfrac{\bigl| W^{\bu{L}}_{\bu{M},\bu{N}} \bigr|}
           {\bigl| \Hom_{  \cal{C}(\cal{A})}(\bu{M},\bu{N}) \bigr|}
     \dfrac{\bigl| \Aut_{\cal{C}(\cal{A})}(\bu{L}) \bigr|}
           {\bigl| \Aut_{\cal{C}(\cal{A})}(\bu{M}) \bigr|
            \bigl| \Aut_{\cal{C}(\cal{A})}(\bu{N}) \bigr|}
     [\bu{M}] \otimes [\bu{N}].
\end{align*}
where $W^{\bu{L}}_{\bu{M},\bu{N}}$ denotes 
the set of all conflations $ \bu{N} \to \bu{L} \to \bu{M}$ 
in $\cal{E}_0$,
and $\chi$ is an arbitrary map from
$\Iso(\cal{C}(\cal{P})) \times \Iso(\cal{C}(\cal{P}))$ 
to $\bb{Z}$ %(or $\bb{C}$)
satisfying the condition \eqref{eq:chi}.
\end{dfn}

\begin{rmk}
As in Remark~\ref{rmk:normalization},
we can rewrite the coproduct in the next equivalent form:
\begin{align*}
\Delta'_{\chi,\cal{E}_0}([[\bu{L}]])
 = \sum_{\bu{M},\bu{N}} 
     t^{\chi(\bu{M},\bu{N})} 
     w^{\bu{L}}_{\bu{M},\bu{N}}
     \dfrac{ a_{\bu{M}} a_{\bu{N}} }{ a_{\bu{L}} }
     [\bu{M}] \otimes [\bu{N}]
\end{align*}
with 
\begin{align}
\label{eq:w}
w^{\bu{L}}_{\bu{M},\bu{N}}:=
     \dfrac{\bigl| W^{\bu{L}}_{\bu{M},\bu{N}} \bigr|}
           {\bigl| \Hom_{  \cal{C}(\cal{A})}(\bu{M},\bu{N}) \bigr|}
     \dfrac{\bigl| \Aut_{\cal{C}(\cal{A})}(\bu{M}) \bigr|
            \bigl| \Aut_{\cal{C}(\cal{A})}(\bu{N}) \bigr|}
           {\bigl| \Aut_{\cal{C}(\cal{A})}(\bu{L}) \bigr|}.
\end{align}
Note that $w^{\bu{L}}_{\bu{M},\bu{N}}$ is a substitute 
of $g^{\bu{L}}_{\bu{M},\bu{N}}$ \eqref{eq:g}.
\end{rmk}

As in Lemma~\ref{lem:coassoc:hcp}, we have
\begin{lem}
Assume that an abelian category $\cal{A}$ satisfies
the conditions (a), (b), (c) and (e).
Then 
$\bigl( \cal{H}(\cal{C}(\cal{P})), \Delta'_{\chi,\cal{E}_0} \bigr)$ 
is a topological coassociative coalgebra over $\bb{C}$.
\end{lem}

Recalling the proof of Lemma~\ref{lem:coassoc:hcp}
we note that this lemma is a consequence 
of the associativity of the algebra 
$\cal{H}_{\tw}(\cal{C}(\cal{P})=
 \bigl( \cal{H}(\cal{C}(\cal{P})),*_{\cal{E}_0} \bigr)$.
\\

Now we want to descend the coproduct $\Delta'_{\chi,\cal{E}_0}$ 
of $\cal{H}(\cal{C}(\cal{P}))$.
to the localized algebra $\cal{DH}(\cal{A})$.
We have 

\begin{prop}
\label{prop:coalg:E_0}
For an abelian category $\cal{A}$ satisfying
the conditions (a), (b), (c) and (e),
the subset 
$$
 S := \left\{ [\bu{M}] \mid H_*(\bu{M})=0 \right\}
$$
of $\cal{H}(\cal{C}(\cal{P}))$ spans a coideal of 
the coalgebra 
$\bigl( \cal{H}(\cal{C}(\cal{P})),
        \Delta'_{\chi,\cal{E}_0} \bigr)$,
\end{prop}

\begin{proof}
By \eqref{eq:acyclic}, 
we can express any element in $S$ as $[\bu{{K_P}} \oplus \bu{{K_Q}}^*]$
with some $P,Q \in \Obj(\cal{P})$.
So we study the short exact sequence
\begin{align}
\label{eq:localize:1}
\xymatrix@1{
  0         \ar[r]
&\bu{N}     \ar[r]
&\bu{{K_P}} \oplus \bu{{K_Q}}^* \ar[r]
&\bu{M}     \ar[r]
&0
}
\end{align}
in $\cal{C}(\cal{P})$.
What should be shown is that 
$[\bu{M}] \mathop{\otimes} [\bu{N}] 
 \in S \mathop{\otimes} C + C \mathop{\otimes} S$
if \eqref{eq:localize:1} appears in $\cal{E}_0$.

The long exact sequence in homology 
induced from \eqref{eq:localize:1}
can be split 
to give two long exact sequences
\begin{align*}
\xymatrix@R=0ex{
 0           \ar[r] 
&K           \ar[r]
&H_1(\bu{N}) \ar[r] 
&0           \ar[r]
&H_1(\bu{M}) \ar[r]
&C           \ar[r] 
&0
\\
 0           \ar[r] 
&C           \ar[r]
&H_0(\bu{N}) \ar[r] 
&0           \ar[r]
&H_0(\bu{M}) \ar[r]
&K           \ar[r] 
&0
}.
\end{align*}
Here we used the $2$-periodicity of complexes.
Thus we have
$H_1(\bu{M})=H_0(\bu{N})$ and 
$H_0(\bu{M})=H_1(\bu{N})$.
Recalling the condition \eqref{eq:E_0:cond},
we conclude that no exact sequence of the form 
\eqref{eq:localize:1}
appears in $\cal{E}_0$.
\end{proof}

%%%%%%%%%%%%%%%%%%%%%%%%%%%%%%%%%%%%%%%%%%%%%%%%%%
%%%%%%%%%%%%%%%%%%%%%%%%%%%%%%%%%%%%%%%%%%%%%%%%%%

\subsection{Genuine coproduct}
\label{subsec:genuine}

Using the formulation of quotient coalgebra 
presented in the last subsection,
we introduce a good coproduct on the whole 
algebra $\cal{DH}(\cal{A})$.

As a preliminary, we have 

\begin{lem}
For an abelian category $\cal{A}$ satisfying
the conditions (a),
any object $\bu{M}$ in $\cal{C}(\cal{P})$ 
is of the form $\bu{M} \cong \bu{N} \oplus \bu{K}$,
where $\bu{N}$ and $\bu{K}$ are objects in $\cal{C}(\cal{P})$ 
and $\bu{K}$ is a maximal acyclic subobject of $\bu{M}$. 
\end{lem}

\begin{proof}
Since the condition (a) ensures 
that $\cal{A}$ enjoys Krull-Schmidt property,
$\cal{C}(\cal{A})$ is also a Krull-Schmidt category.
Then the assertion is trivial.
\end{proof}

By \cite[Lema~3.2]{B} every acyclic object $\bu{K}$ in 
$\cal{C}(\cal{P})$ can be expressed as 
$\bu{{K_P}} \oplus \bu{{K_Q}}^*$.
Recall also that for acyclic $\bu{K}$ and any $\bu{N}$ 
we have 
$\Ext^1_{\cal{C}(\cal{P})}(\bu{N},\bu{K})
=\Ext^1_{\cal{C}(\cal{P})}(\bu{K},\bu{N})=0$.
Combining these results we have 

\begin{lem}
\label{lem:g:basis}
For an abelian category $\cal{A}$ satisfying
the conditions (a) -- (c),
$\cal{DH}(\cal{A})$ has a basis consisting of elements
$$
 [\bu{N}] * K_\alpha * K_\beta^*
$$
with $H^*(\bu{N}) \neq 0$ and $\alpha,\beta \in K(\cal{A})$.
\end{lem}

Now we introduce a coproduct on the whole algebra $\cal{DH}(\cal{A})$.

\begin{dfn}
Let $\cal{A}$ be an abelian category satisfying
the conditions (a), (b), (c) and (e),
and let $\chi$ be an arbitrary map from
$\Iso(\cal{C}(\cal{P})) \times \Iso(\cal{C}(\cal{P}))$ 
to $\bb{Z}$ (or $\bb{C}$)
satisfying the condition \eqref{eq:chi}.

Define a $\bb{C}$-linear map
$$
 \Delta'_{\chi}: 
 \cal{DH}(\cal{A}) \longrightarrow 
 \cal{DH}(\cal{A}) \mathop{\wh{\otimes}}
 \cal{DH}(\cal{A})
$$
by
$$
 \Delta'_{\chi}([\bu{N}] * K_{\alpha} * K_{\beta}^*)
 :=
 \Delta'_{\chi,\cal{E}^0}([\bu{N}])
 * (K_{\alpha} \otimes K_{\alpha})
 * (K_{\beta}^* \otimes K_{\beta}^*).
$$
Here we used Lemma~\ref{lem:g:basis} and the multiplication
$$
 (x \otimes y) * (z \otimes w)
 = (x * z) \otimes (y * w)
$$
on the tensor space 
$\cal{DH}(\cal{A}) \mathop{\wh{\otimes}} \cal{DH}(\cal{A})$.

We also define a $\bb{C}$-linear map
$$
 \ep: \cal{DH}(\cal{A}) \longrightarrow \bb{C}
$$
by
$$
 \ep([\bu{N}] * K_{\alpha} * K_{\beta}^*)
 := \delta_{\bu{N},[0]}.
$$
\end{dfn}

\begin{thm}
\label{thm:coassoc}
Assume that $\cal{A}$ is an abelian category satisfying
the conditions (a) -- (e).
Then the tuple $(\cal{DH}(\cal{A}),\Delta'_{\chi},\ep)$ is a 
topological counital coassociative coalgebra,
\end{thm}

\begin{proof}
Let us denote 
$S:=  \left\{ [\bu{M}] \mid H_*(\bu{M})=0 \right\}$.
Proposition~\ref{prop:coalg:E_0} yields that 
$\Delta'_{\chi,\cal{E}_0}$ descends to 
the quotient $\bb{C}$-vector space
$\cal{H}(\cal{C}(\cal{P}))/\mathop{\text{Span}}(S)$.
By Lemma~\ref{lem:g:basis},
this quotient space has the basis 
consisting of $[\bu{N}]$.
With this observation and 
the Ore condition satisfied by $S$,
we see that $\Delta'_{\chi}$ is well-defined
and  
$(\cal{DH}(\cal{A}),\Delta'_{\chi})$ is 
a topological coassociative coalgebra.

It is easy to check that $\ep$ gives a counit.
\end{proof}

%%%%%%%%%%%%%%%%%%%%%%%%%%%%%%%%%%%%%%%%%%%%%%%%%%
%%%%%%%%%%%%%%%%%%%%%%%%%%%%%%%%%%%%%%%%%%%%%%%%%%
%%%%%%%%%%%%%%%%%%%%%%%%%%%%%%%%%%%%%%%%%%%%%%%%%%

\section{Hereditary case}
\label{sect:hered}

In the case where $\cal{A}$ is hereditary,
one knows that $\cal{H}(\cal{A})$ is embedded into 
$\cal{DH}(\cal{A})$ as an algebra \cite[Lemma~4.3]{B},
and moreover $\cal{DH}(\cal{A})$ is Drinfeld double 
of  $\cal{H}(\cal{A})$ \cite{Y}.
In this section we compare the coalgebra structures 
on $\cal{H}(\cal{A})$ and  $\cal{DH}(\cal{A})$.

%%%%%%%%%%%%%%%%%%%%%%%%%%%%%%%%%%%%%%%%%%%%%%%%%%
%%%%%%%%%%%%%%%%%%%%%%%%%%%%%%%%%%%%%%%%%%%%%%%%%%

\subsection{Basis of $\cal{DH}(\cal{A})$}

Assume that 
$\cal{A}$ satisfies the conditions (a)--(e).
By \cite[\S4]{B} we have a nice basis for $\cal{DH}(\cal{A})$.
To explain that, let us recall the minimal resolution of objects of $\cal{A}$.

\begin{dfn}[{\cite[\S4]{B}}]
\label{dfn:B:4}
Assume the conditions (a), (c) and (d) on $\cal{A}$. 

\begin{enumerate}
\item[(1)]
Every object $A\in \cal{A}$ has a projective resolution
\begin{align}\label{eq:res}
\xymatrix@1{
 0 \ar[r] 
&P \ar[r]^{f} 
&Q \ar[r]
&A \ar[r]
&0}.
\end{align}
Decomposing $P$ and $Q$ into finite direct sums
$P=\oplus_i P_i$, $Q=\oplus_j Q_j$,
one may write $f=(f_{ij})$ in matrix form with $f_{ij}: P_i\to Q_j$. 
The resolution \eqref{eq:res} is said to be minimal 
if none of the morphisms $f_{ij}$ is an isomorphism. 

\item[(2)]
Given an object $A$ in $\cal{A}$, 
take a minimal projective resolution 
\begin{align*}
\xymatrix@1{
 0 \ar[r] 
&P_A \ar[r]^{f_A} 
&Q_A \ar[r] 
&A \ar[r]
&0},
\end{align*}
we define a $\bb{Z}/2\bb{Z}$-graded complex
\begin{align}
\label{eq:C_A}
\bu{{C_A}} 
 := \cpx{P_A}{f_A}{0}{Q_A} \in \Obj(\cal{C}(\cal{P})).
\end{align}
(Remark: By \cite[Lemma~4.1]{B}, 
arbitrary two minimal projective resolutions of $A$ are isomorphic,
so the complex $\bu{{C_A}}$ is well-defined up to isomorphism.)

\item[(3)]
Assume $\cal{A}$ satisfies the conditions (a) -- (e). 
Given an object $A$ in $\cal{A}$, 
we define elements  $\bu{{E_A}},\bu{{F_A}}$ in $\cal{DH}(\cal{A})$ by
\begin{align}
\label{eq:E_A}
E_A :=  t^{\<P_A,A\>} K_{-\cl{P_A}}*[\bu{{C_A}}],
\qquad
F_A :=  {E_A}^*.
\end{align}
Here we used a minimal projective decomposition of $A$
and the associated complex $\bu{{C_A}}$ shown in \eqref{eq:C_A}.
\end{enumerate}
\end{dfn}

%\begin{rmk}
%The complex $\bu{{C_A}}$ is denoted as $C_A$ in \cite{B}.
%\end{rmk}

\begin{fct}[{\cite[Lemma\,4.2]{B}}]
\label{fct:B:lem4.2}
Assume $\cal{A}$ is an abelian category
satisfying the conditions (a), (c) and (d).
Then every object $\bu{M}$ in $C(\cal{P})$ has a direct sum decomposition
\begin{align*}
\bu{M} = \bu{{C_A}} \oplus \bu{{C_B}}^* 
          \oplus \bu{{K_P}} \oplus \bu{{K_Q}}^*.
\end{align*}
Moreover, 
the objects $A,B \in \Obj(\cal{A})$ and $P,Q \in \Obj(\cal{P})$ 
are unique up to isomorphism.
\end{fct}

One also has

\begin{fct}[{Corollary of \cite[Lemmas~4.6, 4.7]{B}}]
\label{fct:basis}
$\cal{DH}(\cal{A})$ has a basis consisting of elements 
\begin{align*}
E_A * K_\alpha * K_{\beta}^* * F_B,\quad
A,B \in \Iso(\cal{A}),\ \alpha,\beta \in K(\cal{A}).
\end{align*}
\end{fct}

%%%%%%%%%%%%%%%%%%%%%%%%%%%%%%%%%%%%%%%%%%%%%%%%%%
%%%%%%%%%%%%%%%%%%%%%%%%%%%%%%%%%%%%%%%%%%%%%%%%%%

\subsection{Twisted coproduct and coalgebra embedding}

Let us recall the  twisted Hall algebra $\cal{H}_{\tw}(\cal{A})$ 
for an abelian category $\cal{A}$.

\begin{dfn}
Let $\cal{A}$ be an abelian category 
satisfying the conditions (a) --(c). 
The twisted Hall algebra 
$\cal{H}_{\tw}(\cal{A})=\bigl( \cal{H}(\cal{A}),*,[0] \bigr)$ 
is the tuple consisting of the $\bb{C}$-vector space $\cal{H}(\cal{A})$,  
the twisted multiplication
\begin{equation}
\label{eq:tw-mul2}
[A]*[B] := t^{ \< \cl{A},\cl{B} \> } [A] \diamond [B] 
\end{equation}
for $A,B \in \Iso(\cal{A})$,
and the class of the zero object.
\end{dfn}

\begin{rmk}
As mentioned in \cite[\S3.5]{B},
the symbol $*$ has different meanings 
in $\cal{H}_{\tw}(\cal{A})$ and $\cal{H}_{\tw}(\cal{C}(\cal{P}))$.
Compare the expressions \eqref{eq:tw-mul} and \eqref{eq:tw-mul2}.
\end{rmk}

As for the relation to Bridgeland's Hall algebra, we have 

\begin{fct}[{\cite[Lemma~4.3]{B}}]
\label{fct:B:lem4.3}
Assume that $\cal{A}$ satisfies the conditions (a) -- (e).
Then there is an embedding of $\bb{C}$-algebras 
\begin{align}
\label{eq:I^e_+:1}
I^{\e}_+: 
 \bigl( \cal{H}_{\tw}(\cal{A}),*,[0] \bigr) 
 \longinto
 \bigl( \cal{DH}(\cal{A}),*,[0] \bigr) 
\qquad
[A] \longmapsto E_A.
\end{align}
\end{fct}

Next let us recall 
the extended Hall algebra $\cal{H}^{\e}_{\tw}(\cal{A})$
and the twisted coproduct $\Delta$ on it
(see also \cite[\S1.5, Page 16]{S} for the explanation).

The extended Hall algebra $\cal{H}_{\tw}^{\e}(\cal{A})$ 
is defined as an extension of $\cal{H}_{\tw}(\cal{A})$
by adjoining symbols $K_\alpha$ for classes $\alpha\in K(\cal{A})$, 
and imposing relations
\begin{align*}
K_\alpha * K_\beta = K_{\alpha+\beta}, \qquad 
K_\alpha * [B]     = t^{( \alpha,\cl{B})} [B] * K_\alpha
\end{align*}
for $\alpha,\beta \in K(\cal{A})$ 
and $B \in \Iso(\cal{A})$.
Here we used the symmetrized Euler form \eqref{eq:sym-euler}.
Thus $\cal{H}_{\tw}^{\e}(\cal{A})$ 
has a vector space basis consisting of the elements 
$K_\alpha * [B]$ for $\alpha\in K(\cal{A})$ and $B \in \Iso(\cal{A})$.

\begin{fct}
Assume that  $\cal{A}$ satisfies 
the conditions (a), (b), (c) and (e).
\begin{enumerate}
\item[(1)]
Then the maps
\begin{align}
\nonumber
&\Delta_{\cal{H}^{\e}_{\tw}(\cal{A})} \equiv \Delta
 : \cal{H}_{\tw}^{\e}(\cal{A}) 
         \longrightarrow 
         \cal{H}_{\tw}^{\e}(\cal{A}) \ctp 
         \cal{H}_{\tw}^{\e}(\cal{A}), 
\\
\label{eq:DHetw}
&\Delta([A] * K_\alpha ) :=
  \sum_{B,C} t^{\<\cl{B},\cl{C}\>} 
  \dfrac{\bigl| \Ext^1_{\cal{A}}(B,C)_A \bigr|}
        {\bigl| \Hom_{  \cal{A}}(B,C)   \bigr|}
  \dfrac{a_A}{a_B a_C}
   \bigl( [B] * K_{\widehat{C}+\alpha} \bigr) \otimes 
   \bigl( [C] * K_\alpha \bigr),
\end{align}
and
\begin{align*}
\ep: \cal{H}_{\tw}^{\e}(\cal{A}) 
     \longrightarrow \bb{C}, \qquad
\ep([A] * K_{\alpha}):=\delta_{A,0},
\end{align*} 
define a topological counital coassociative coalgebra structure
on $\cal{H}_{\tw}^{\e}(\cal{A})$.

\item[(2)]
Assume moreover that $\cal{A}$ also satisfies 
the condition (d).
Then the tuple
$(\cal{H}_{\tw}^{\e}(\cal{A}),*,[0],\Delta,\ep)$,
is a topological bialgebra defined over $\bb{C}$.
\end{enumerate}
\end{fct}

The relation to Bridgeland's Hall algebra is described by 

\begin{fct}[{\cite[Lemma~4.6]{B}}]
\label{fct:B:lem4.6}
Assume that $\cal{A}$ satisfies the conditions (a) -- (e).
Then there is an embedding of algebras 
\begin{align}
\label{eq:I^e_+}
I^{\e}_+: 
 \bigl( \cal{H}^{\e}_{\tw}(\cal{A}),*,[0]) 
 \longinto 
 \bigl( \cal{DH}(\cal{A}),*,[0])
\end{align}
defined on generators by $[A] \mapsto E_A$ and 
$K_\alpha \mapsto K_{\alpha}$. 
\end{fct}

Now we introduce a twisted coproduct on $\cal{DH}(\cal{A})$
using the basis shown in Fact~\ref{fct:B:lem4.2} 
and Fact~\ref{fct:basis}.

\begin{dfn}
\label{dfn:coprod}
Let $\cal{A}$ be an abelian category satisfying
the conditions (a) -- (e),
and let $\chi$ be an arbitrary map from
$\Iso(\cal{C}(\cal{P})) \times \Iso(\cal{C}(\cal{P}))$ 
to $\bb{Z}$ (or $\bb{C}$)
satisfying the condition \eqref{eq:chi}.

Define a $\bb{C}$-linear map
$$
 \Delta_{\chi}: 
 \cal{DH}(\cal{A}) \longrightarrow 
 \cal{DH}(\cal{A}) \mathop{\wh{\otimes}}
 \cal{DH}(\cal{A})
$$
by
$$
 \Delta_{\chi}(
  [\bu{{C_A}} \oplus \bu{{C_B}}^*] 
  * K_{\alpha} * K_{\beta}^*)
 :=
 \Delta_{\chi,\cal{E}^0}([\bu{{C_A}}])
 \Delta^{op}_{\chi,\cal{E}^0}([\bu{{C_B}}])
 * (K_{\alpha} \otimes K_{\alpha})
 * (K_{\beta}^* \otimes K_{\beta}^*).
$$
with
\begin{align*}
 \Delta_{\chi,\cal{E}^0}([\bu{L}]) := 
& \sum_{\bu{M},\bu{N}} 
     t^{ \chi(\bu{M},\bu{N}) } 
     \dfrac{\bigl| W^{ \bu{L} }_{\bu{M},\bu{N}} \bigr|}
           {\bigl|\Hom_{  \cal{C}(\cal{A})}(\bu{M},\bu{N}) \bigr|}
     \dfrac{a_{\bu{L}}}{ a_{\bu{M}} a_{\bu{N}} }
         \bigl( K_{\cl{N_0}} * [\bu{M}] \bigr) \otimes 
         \bigl( [\bu{N}] * K_{\cl{M_1}} \bigr),
\\
 \Delta^{op}_{\chi,\cal{E}^0}([\bu{L}]) := 
& \sum_{\bu{M},\bu{N}} 
     t^{ \chi(\bu{M},\bu{N}) } 
     \dfrac{\bigl| W^{ \bu{L} }_{\bu{M},\bu{N}} \bigr|}
           {\bigl|\Hom_{  \cal{C}(\cal{A})}(\bu{M},\bu{N}) \bigr|}
     \dfrac{a_{\bu{L}}}{ a_{\bu{M}} a_{\bu{N}} }
         \bigl( K^*_{\cl{N_1}} * [\bu{M}] \bigr) \otimes 
         \bigl( [\bu{N}] * K^*_{\cl{M_0}} \bigr),
\end{align*}
where $W^{\bu{L}}_{\bu{M},\bu{N}}$ denotes 
the set of all conflations $ \bu{N} \to \bu{L} \to \bu{M}$ 
in $\cal{E}_0$,
and we used the multiplication
$$
 (x \otimes y) * (z \otimes w)
 = (x * z) \otimes (y * w)
$$
on the tensor space 
$\cal{DH}(\cal{A}) \mathop{\wh{\otimes}} \cal{DH}(\cal{A})$.
\end{dfn}

Then we have 

\begin{prop}
For an abelian category $\cal{A}$ satisfying
the conditions (a), (b), (c) and (e),
the tuple $(\cal{DH}(\cal{A}),\Delta_{\chi},\ep)$ is a 
topological coassociative coalgebra.
\end{prop}

\begin{proof}
Let us show 
$(\Delta \otimes 1)\circ\Delta([[\bu{L}]])
=(1 \otimes \Delta)\circ\Delta([[\bu{L}]])$
for $\bu{L}=\bu{{C_A}}$.
The other cases are similar.
As in the proof of Lemma~\ref{lem:coassoc:hcp},
it is equivalent to the formula
\begin{align*}
&\sum_{\bu{M}}
  t^{ \chi(\bu{M},\bu{N}) + \chi(\bu{P},\bu{Q})}
  w^{\bu{L}}_{\bu{M},\bu{N}} w^{\bu{M}}_{\bu{P},\bu{Q}}
  \dfrac{ a_{\bu{N}} a_{\bu{P}} a_{\bu{Q}} }{ a_{\bu{L}} }
\\
&\phantom{\sum_{\bu{M}} t}
  \bigl( K_{\cl{N_0}} * K_{\cl{Q_0}} * [[\bu{P}]]   \bigr) \otimes 
  \bigl( K_{\cl{N_0}} * [[\bu{Q}]]   * K_{\cl{P_1}} \bigr) \otimes 
  \bigl( [[\bu{N}]]   * K_{\cl{M_1}} \bigr)
\\
=
&\sum_{\bu{R}}
  t^{  \chi(\bu{P},\bu{R}) + \chi(\bu{Q},\bu{N}) }
  w^{\bu{L}}_{\bu{P},\bu{R}}  w^{\bu{R}}_{\bu{Q},\bu{N}}
  \dfrac{ a_{\bu{N}} a_{\bu{P}} a_{\bu{Q}} }{ a_{\bu{L}} }
\\
&\phantom{\sum_{\bu{M}} t}
  \bigl( K_{\cl{R_0}} * [[\bu{P}]]                  \bigr) \otimes 
  \bigl( K_{\cl{N_0}} * [[\bu{Q}]]   * K_{\cl{P_1}} \bigr) \otimes 
  \bigl( [[\bu{N}]]   * K_{\cl{Q_1}} * K_{\cl{P_1}} \bigr)
\end{align*}
with 
$$
w^{\bu{L}}_{\bu{M},\bu{N}}:=
     \dfrac{\bigl| W^{\bu{L}}_{\bu{M},\bu{N}} \bigr|}
           {\bigl| \Hom_{  \cal{C}(\cal{A})}(\bu{M},\bu{N}) \bigr|}
     \dfrac{\bigl| \Aut_{\cal{C}(\cal{A})}(\bu{M}) \bigr|
            \bigl| \Aut_{\cal{C}(\cal{A})}(\bu{N}) \bigr|}
           {\bigl| \Aut_{\cal{C}(\cal{A})}(\bu{L}) \bigr|}
$$
for any tuple $(\bu{L},\bu{N},\bu{P},\bu{Q})$ of 
objects in $\cal{C}(\cal{P})$ 
satisfying the condition 
$\cl{M_i}=\cl{P_i}+\cl{Q_i}$ and $\cl{R_i}=\cl{Q_i}+\cl{N_i}$ 
for $i=0,1$.
This condition tells us that the coassociativity follows 
from the formula
$\sum_{\bu{M}}
  w^{\bu{L}}_{\bu{M},\bu{N}} w^{\bu{M}}_{\bu{P},\bu{Q}}
=\sum_{\bu{R}} 
 w^{\bu{L}}_{\bu{P},\bu{R}} w^{\bu{R}}_{\bu{Q},\bu{N}}$,
which is the associativity 
of $\cal{H}(\cal{C}(\cal{P}),\cal{E}_0)$ .
\end{proof}

Our claim is 

\begin{thm}
\label{thm:embedding}
For an abelian category $\cal{A}$ satisfying the conditions (a) -- (e),
the map $I^{\e}_+$ \eqref{eq:I^e_+} defines an embedding of 
$\bb{C}$-coalgebras
$$
I^{\e}_+: 
 \bigl( \cal{H}^{\e}_{\tw}(\cal{A}),\Delta_{\cal{H}(\cal{A})},\ep \bigr) 
  \longinto 
 \bigl( \cal{DH}(\cal{A}),\Delta_{\chi_0},\ep \bigr).
$$
Here we use 
$$
 I^{\e}_+(x \otimes y) := I^{\e}_+(x) \otimes I^{\e}_+(y)
$$
on the tensor product space 
$\cal{H}^{\e}_{\tw}(\cal{A}) \ctp \cal{H}^{\e}_{\tw}(\cal{A})$
and 
$$
\chi_0(\bu{M},\bu{N}) := -\<\bu{N},\bu{M}\>.
$$
\end{thm}

\begin{rmk}
\begin{enumerate}
\item
As mentioned in Remark~\ref{rmk:coalg} (1),
$\chi=\chi_0$ satisfies the condition \eqref{eq:chi}
imposed on the map $\chi$.

\item
As in Remark~\ref{rmk:coalg}, 
one can rewrite the coproduct $\Delta_{\chi_0}$ using the notation 
$[[\bu{L}]] = [\bu{L}]/a_{\bu{L}}$ into the form 
\begin{align}
\label{eq:delta'}
% \Delta_{\chi_0}([\bu{L}])
% := \sum_{\bu{M},\bu{N}} 
%     t^{-\<\bu{N},\bu{M}\>'} 
%     \dfrac{\bigl| W^{\bu{L}}_{\bu{M},\bu{N}} \bigr|}
%           {\bigl| \Hom_{  \cal{C}(\cal{A})}(\bu{M},\bu{N}) \bigr|}
%     \dfrac{ a_{\bu{L}} }{ a_{\bu{M}} a_{\bu{N}} }
%     [\bu{M}] \otimes [\bu{N}].
\Delta([[\bu{L}]])
 = \sum_{\bu{M},\bu{N}}  
    t^{ -\<\bu{N},\bu{M}\>' }
    w^{\bu{L}}_{\bu{M},\bu{N}} 
    \dfrac{ a_{\bu{M}} a_{\bu{N}} }{ a_{\bu{L}} }
     \bigl( K_{\cl{N_0}} * [[\bu{M}]]  \bigr) \otimes 
     \bigl( [[\bu{N}]] * K_{\cl{M_1}} \bigr).
\end{align}
\end{enumerate}
\end{rmk}

The proof of the theorem is presented in the next subsection.
Since our construction is symmetric with respect to the involution $*$,
we also have 

\begin{thm}
\label{thm:embedding2}
For an abelian category $\cal{A}$ satisfying the conditions (a) -- (e),
the map 
$$
 I^{\e}_-: 
 \cal{H}^{\e}_{\tw}(\cal{A}) \longinto \cal{DH}(\cal{A}),
 \qquad
 [B] \longmapsto F_B,\ K_{\beta} \longmapsto K^*_\beta
$$
defines an embedding of $\bb{C}$-coalgebras
$$
I^{\e}_-: 
 \bigl( \cal{H}^{\e}_{\tw}(\cal{A}),\Delta_{\cal{H}(\cal{A})},\ep \bigr) 
  \longinto 
 \bigl( \cal{DH}(\cal{A}),\Delta_{\chi_0},\ep \bigr).
$$
\end{thm}

%%%%%%%%%%%%%%%%%%%%%%%%%%%%%%%%%%%%%%%%%%%%%%%%%%
%%%%%%%%%%%%%%%%%%%%%%%%%%%%%%%%%%%%%%%%%%%%%%%%%%

\subsection{Proof of Theorem~\ref{thm:embedding}}

We begin with introducing several lemmas.
Assume that $\cal{A}$ satisfies the conditions (a) -- (e).
In the argument we will use the symbols $P_A,Q_A$ 
introduced in Definition~\ref{dfn:B:4}.

\begin{lem}
\label{lem:NCM}
For an exact sequence
\begin{align}
\label{eq:NCM}
\xymatrix@1{
 0 \ar[r]
&\bu{N} \ar[r]
&\bu{{C_A}} \ar[r]
&\bu{M} \ar[r]
&0
} 
\end{align}
in $\cal{E}_0$
with $A \in \Obj(\cal{A})$,
$\bu{M}$ and $\bu{N}$ are of the form $\bu{{C_{X}}}$
with some $X \in \Obj(\cal{A})$.
Moreover,
if $w^{\bu{{C_A}}}_{\bu{M},\bu{N}} \neq 0$ 
then one can express $\bu{M}=\bu{{C_B}}$, 
$\bu{N}=\bu{{C_D}}$ 
and 
$w^{\bu{{C_A}}}_{\bu{M},\bu{N}}=
 g^{\bu{{C_A}}}_{\bu{{C_B}},\bu{{C_D}}}$.
\end{lem}

\begin{proof}
Let us express the exact sequence \eqref{eq:NCM}
in the following exact commutative diagram:
\begin{align*}
\xymatrix{
 0   \ar[r]    
&N_1 \ar[r] \ar@<-.4ex>[d]_{d^N_1}
&P_A \ar[r] \ar@{_{(}->}@<-.4ex>[d]_{f_A}
&M_1 \ar[r] \ar@<-.4ex>[d]_{d^M_1}
&0
\\
 0   \ar[r] 
&N_0 \ar[r] \ar@<-.4ex>[u]_{d^N_0}
&Q_A \ar[r] \ar@<-.4ex>[u]_{0}
&M_0 \ar[r] \ar@<-.4ex>[u]_{d^M_0}
&0
}. 
\end{align*}
Then it is easily seen that 
$d^N_0 = 0$ and $d^M_0 = 0$.
By the snake lemma we have
the long exact sequence
$$
\xymatrix@1{
 0             \ar[r]
&\Ker(d^N_1)   \ar[r]
&0             \ar[r]
&\Ker(d^M_1)   \ar[r]
&\Coker(d^N_1) \ar[r]
&A             \ar[r]
&\Coker(d^M_1) \ar[r]
&0
}.
$$
Thus $\Ker(d^N_1)=0$,
which with Fact~\ref{fct:B:lem4.2} implies 
$\bu{N} = \bu{{C_D}} \oplus \bu{K}$
with $D\in\Obj(\cal{A})$ and $H_*(\bu{K})=0$.
Since $\bu{N}$ is a subcomplex of $\bu{{C_A}}$, 
$\bu{K}$ appears in the resolution $\bu{{C_A}}$.
Since $\bu{{C_A}}$ is minimal,
we have $\bu{K}=0$.

Next recall the condition \eqref{eq:E_0:cond} of $\cal{E}_0$.
If $H_1(\bu{M}) \neq 0$ then $H_0(\bu{N})=0$,
but since $H_0(\bu{N})=D$ we have $D=0$ and $\bu{N}=0$.
Then $\bu{M}=\bu{{C_A}}$ and it is done.

If $H_1(\bu{M}) = 0$ then we have 
$\bu{M} = \bu{{C_B}} \oplus \bu{K}$ 
with $B\in\Obj(\cal{A})$ and $H_*(\bu{K})=0$.
Also in this case the minimality of the resolution $\bu{{C_A}}$
implies $\bu{K}=0$, 
so we are done.
\end{proof}

\begin{lem}
\label{lem:aa}
For any $A,B\in \Obj(\cal{A})$ we have
$$
a_{\bu{{C_A}}} =a_{A} \bigl| \Hom_{\cal{A}}(Q_A,P_A) \bigr|,
\qquad
\bigl| \Hom_{\cal{C}(\cal{A})}(\bu{{C_A}},\bu{{C_B}} ) \bigr|
=
\bigl| \Hom_{\cal{A}}(Q_A,P_B) \bigr|
\bigl| \Hom_{\cal{A}}(A, B) \bigr|.
$$
\end{lem}

\begin{proof}
As mentioned in \cite[Proof of Lemma~4.3]{B},
one can easily see that
there is a short exact sequence
$$
\xymatrix@1{
 0\ar[r]
&\Hom_{\cal{A}}(Q_A,P_B ) \ar[r]
&\Hom_{\cal{C}(\cal{A})}(\bu{{C_A}},\bu{{C_B}} ) \ar[r]
&\Hom_{\cal{A}}(A,B) \ar[r]
&0
},
$$
which yields the assertions.
\end{proof}

\begin{lem}
\label{lem:ggt}
For $A,B,D \in \Obj(\cal{A})$ we have 
$$
 g^{\bu{{C_A}}}_{\bu{{C_B}},\bu{{C_D}}}/g^A_{B,D}
=t^{2\<Q_D,P_B\>}.
$$
\end{lem}

\begin{proof}
Using the formula \eqref{eq:gehaaa},
$\bigl| \Ext_{\cal{A}}(B,D)_A \bigr| 
=\bigl| \Ext_{\cal{C}(\cal{A})}(\bu{{C_B}},\bu{{C_D}})_{\bu{{C_A}}} \bigr|$ 
and Lemma~\ref{lem:aa}, we have
\begin{align*}
\dfrac{g^{\bu{{C_A}}}_{\bu{{C_B}},\bu{{C_D}}}}{g^A_{B,D}}
&=
 \dfrac{\bigl| \Hom_{\cal{A}}(B,D) \bigr|}
       {\bigl| \Hom_{\cal{C}(\cal{A})}(\bu{{C_B}},\bu{{C_D}}) \bigr|}
 \dfrac{ a_{\bu{{C_A}}} }{a_{A}}
 \dfrac{a_{B}}{ a_{\bu{{C_B}}} }
 \dfrac{a_{C}}{ a_{\bu{{C_C}}} }
\\
&=
 \dfrac{1}{\bigl| \Hom_{\cal{A}}(Q_B,P_D) \bigr|}
 \dfrac{\bigl| \Hom_{\cal{A}}(Q_A,P_A) \bigr|}
       {\bigl| \Hom_{\cal{A}}(Q_B,P_B) \bigr|
        \bigl| \Hom_{\cal{A}}(Q_D,P_D) \bigr|}.
\end{align*}
Note that the exact sequences
$0 \to {C_D}_i \to {C_A}_i \to {C_B}_i \to 0$ ($i=0,1$)
split since ${C_B}_i$ are projective.
So ${C_A}_i \cong {C_D}_i \oplus {C_B}_i$,
which yields
\begin{align*}
\dfrac{g^{\bu{{C_A}}}_{\bu{{C_B}},\bu{{C_D}}}}{g^A_{B,D}}
&=
 \dfrac{1}{\bigl| \Hom_{\cal{A}}(Q_B,P_D) \bigr|}
 \dfrac{\bigl| \Hom_{\cal{A}}(Q_B \oplus Q_C,P_B \oplus P_D) \bigr|}
       {\bigl| \Hom_{\cal{A}}(Q_B,P_B) \bigr|
        \bigl| \Hom_{\cal{A}}(Q_D,P_D) \bigr|}
\\
&= \bigl| \Hom_{\cal{A}}(Q_B,P_D) \bigr|^{-1}
   \bigl| \Hom_{\cal{A}}(Q_B,P_D) \bigr|
   \bigl| \Hom_{\cal{A}}(Q_D,P_B) \bigr|
\\
&= \bigl| \Hom_{\cal{A}}(Q_D,P_B) \bigr|.
\end{align*}
The last equation is equal to 
$t^{2\<Q_D,P_B\>}$,
since by  the hereditarity of $\cal{A}$ we have
$$
 t^{2\<Q_D,P_B\>}
=q^{\<Q_D,P_B\>}
=\bigl| \Hom_{\cal{A}}(Q_D,P_B) \bigr|/\bigl| \Ext_{\cal{A}}(Q_D,P_B) \bigr|
$$
and we also have 
$\bigl| \Ext_{\cal{A}}(Q_D,P_B) \bigr|=1$
by $P_B \in \Obj(\cal{P})$.
\end{proof}

\begin{lem}
\label{lem:DcCA}
For $A \in \Obj(\cal{A})$ we have
$$
\Delta_{\chi_0}([\bu{{C_A}}])
=\sum_{B,D \in \Iso(\cal{C}(\cal{A}))} 
   t^{\<D,P_B\>-\<Q_D,B\>}  g^{A}_{B, D}
    \bigl( K_{\cl{Q_D}} * [\bu{{C_B}}]   \bigr) \otimes 
    \bigl( [\bu{{C_D}}]   * K_{\cl{P_B}} \bigr).
$$
\end{lem}

\begin{proof}
By Lemma~\ref{lem:NCM} and the formula \eqref{eq:gehaaa} we have 
\begin{align*}
\Delta_{\chi_0}([\bu{{C_A}}])
&= \sum_{\bu{M},\bu{N} \in \Iso(\cal{C}(\cal{P}))} 
     t^{-\<\bu{N},\bu{M}\>'} 
     \dfrac{\bigl| W^{\bu{{C_A}}}_{\bu{M},\bu{N}} \bigr|}
           {\bigl| \Hom_{  \cal{C}(\cal{A})}(\bu{M},\bu{N}) \bigr|}
     \dfrac{ a_{\bu{{C_A}}} } { a_{\bu{M}} a_{\bu{N}} }
     \bigl( K_{\cl{N_0}} * [\bu{M}]   \bigr) \otimes 
     \bigl( [\bu{N}]   * K_{\cl{M_1}} \bigr)
\\
&=\sum_{B,D  \in \Iso(\cal{A})} 
   t^{-\<\bu{{C_D}},\bu{{C_B}}\>'} 
    \dfrac{\bigl| \Ext^1_{\cal{C}(\cal{A})}
                 (\bu{{C_B}},\bu{{C_D}})_{\bu{{C_A}}} \bigr|}
         {\bigl| \Hom_{  \cal{C}(\cal{A})}(\bu{{C_B}},\bu{{C_D}}) \bigr|}
    \dfrac{ a_{\bu{{C_A}}} } { a_{\bu{{C_B}}} a_{\bu{{C_D}}} }
    \bigl( K_{\cl{Q_D}} * [\bu{{C_B}}]   \bigr) \otimes 
    \bigl( [\bu{{C_D}}] * K_{\cl{P_B}} \bigr)
\\
&=\sum_{B,D} 
   t^{-\<\bu{{C_D}},\bu{{C_B}}\>'} 
     g^{ \bu{{C_A}} }_{\bu{{C_B}}, \bu{{C_D}} }
    \bigl( K_{\cl{Q_D}} * [\bu{{C_B}}]   \bigr) \otimes 
    \bigl( [\bu{{C_D}}]   * K_{\cl{P_B}} \bigr)
\end{align*}
Then by Lemma~\ref{lem:ggt} we have
\begin{align*}
\Delta_{\chi_0}([\bu{{C_A}}])
&=\sum_{B,D} 
   t^{-\<\bu{{C_D}},\bu{{C_B}}\>'+2\<Q_D,P_B\>}  g^{A}_{B, D}
    \bigl( K_{\cl{Q_D}} * [\bu{{C_B}}]   \bigr) \otimes 
    \bigl( [\bu{{C_D}}]   * K_{\cl{P_B}} \bigr).
\end{align*}
Note that 
\begin{align*} 
  -\<\bu{{C_D}},\bu{{C_B}}\>'+2\<Q_D,P_B\>
=2\<Q_D,P_B\>-\<Q_D,Q_B\>-\<P_D,P_B\>
=\<D,P_B\>-\<Q_D,B\>
\end{align*}
by $\cl{X}=\cl{Q_X}-\cl{P_X}$.
Thus we have the conclusion.
\end{proof}

Now we start 

\begin{proof}[{Proof of  Theorem~\ref{thm:embedding}}]
By Fact~\ref{fct:B:lem4.6}, it is enough to show 
\begin{align}
\label{eq:compatible1:1}
 I^{\e}_+ \circ \Delta([A]) = \Delta_{\chi_0}(E_A)
\end{align}
for any $A \in \Obj(\cal{A})$.
Recalling the definition \eqref{eq:DHetw} 
of $\Delta \equiv \Delta_{\cal{H}^{\e}_{\tw}(\cal{A})}$ 
on $\cal{H}^{\e}_{\tw}(\cal{A})$, we have
\begin{align*}
\text{LHS of \eqref{eq:compatible1:1}}
&= I^{\e}_+ \circ \Delta'(a_A [[A]])
 = I^{\e}_+ 
    \Bigl(
     \sum_{B,D} t^{\<B,D\>} g^A_{B,D} 
     \bigl( [B] * K_{\cl{D}} \bigr) \otimes [D] 
    \Bigr)
\\
&= \sum_{B,D} t^{\<B,D\>} 
     g^A_{B,D} \bigl( E_B * K_{\cl{D}} \bigr) \otimes E_D.
\end{align*}
Using the definition \eqref{eq:E_A} of $E_A$, 
we have
\begin{align*}
\text{LHS of \eqref{eq:compatible1:1}}
= \sum_{B,D} t^{\<B,D\>} g^A_{B,D} 
   \bigl( t^{\<P_B,B\>} K_{-\cl{P_B}} * [\bu{{C_B}}] * K_{\cl{D}} \bigr) 
   \otimes 
   \bigl( t^{\<P_D,D\>} K_{-\cl{P_D}} * [\bu{{C_D}}] \bigr).
\end{align*}
Recalling the relation \eqref{eq:KM} 
and $\cl{\bu{{C_X}}}=\cl{X}$, we have
\begin{align}
\nonumber
\text{LHS of \eqref{eq:compatible1:1}}
&= \sum_{B,D} 
   t^{\<B,D\> + \<P_B,B\> + \<P_D,D\> 
      + (-\cl{P_D},\cl{\bu{{C_D}}}) - (\cl{D},\cl{\bu{{C_B}}}) }
   g^A_{B,D} 
   \bigl(K_{-\cl{P_B}}  * K_{\cl{D}} * [\bu{{C_B}}] \bigr) 
   \otimes 
   \bigl([\bu{{C_D}}] * K_{-\cl{P_D}} \bigr)
\\
\nonumber
&=\sum_{B,D} 
   t^{\<B,D\> + \<P_B,B\> + \<P_D,D\> - (P_D,D) - (D,B)}
   g^A_{B,D} 
   \bigl(K_{-\cl{P_B}} * K_{\cl{D}} * [\bu{{C_B}}] \bigr) 
   \otimes 
   \bigl([\bu{{C_D}}] * K_{-\cl{P_D}} \bigr)
\\
\label{eq:embed:L}
&=\sum_{B,D} 
   t^{\<P_B,B\> - \<D,P_D\> - \<D,B\> }
   g^A_{B,D} 
   \bigl(K_{-\cl{P_B}} * K_{\cl{D}} * [\bu{{C_B}}] \bigr) 
   \otimes 
   \bigl([\bu{{C_D}}] * K_{-\cl{P_D}} \bigr).
\end{align}

On the other hand,
using Lemma~\ref{lem:DcCA} we have
\begin{align*}
\text{RHS of \eqref{eq:compatible1:1}}
&=\Delta_{\chi}(t^{\<P_A,A\>}K_{-\cl{P_A}}*[\bu{{C_A}}])
\\
&=t^{\<P_A,A\> }
   \sum_{B,D} 
   t^{\<D,P_B\>-\<Q_D,B\>}  g^{A}_{B, D}
    \bigl(K_{-\cl{P_A}} * K_{\cl{Q_D}} * [\bu{{C_B}}]   \bigr) 
    \mathop{\otimes} 
    \bigl(K_{-\cl{P_A}} * [\bu{{C_D}}] * K_{\cl{P_B}} \bigr).
\end{align*}
Using the formula \eqref{eq:KM}
we have
\begin{align}
\nonumber
&\text{RHS of \eqref{eq:compatible1:1}}
\\
\nonumber
&=\sum_{B,D} 
   t^{\<P_A,A\>+\<D,P_B\>-\<Q_D,B\>+(-\cl{P_A},\cl{\bu{{C_D}}})\>}  
   g^{A}_{B, D}
   \bigl(K_{-\cl{P_A}} * K_{\cl{Q_D}} * [\bu{{C_B}}] \bigr) 
   \mathop{\otimes} 
   \bigl([\bu{{C_D}}] * K_{-\cl{P_A}} * K_{\cl{P_B}} \bigr)
\\
\label{eq:embed:R}
&=\sum_{B,D} 
   t^{\<P_A,A\>+\<D,P_B\>-\<Q_D,B\>-(P_A,D)\>}  
   g^{A}_{B, D}
   \bigl(K_{-\cl{P_A}} * K_{\cl{Q_D}} * [\bu{{C_B}}] \bigr) 
   \mathop{\otimes} 
   \bigl([\bu{{C_D}}] * K_{-\cl{P_A}} * K_{\cl{P_B}} \bigr).
\end{align}

Now we compare \eqref{eq:embed:L} and \eqref{eq:embed:R}.
We can assume $\cl{A}=\cl{B}+\cl{D}$
since otherwise $g^A_{B,D}=0$.
Since we are considering the exact commutative diagram
\begin{align*}
\xymatrix{
 0   \ar[r]    
&P_B \ar[r] \ar@{_{(}->}@<-.4ex>[d]_{f_B}
&P_A \ar[r] \ar@{_{(}->}@<-.4ex>[d]_{f_A}
&P_D \ar[r] \ar@{_{(}->}@<-.4ex>[d]_{f_D}
&0
\\
 0   \ar[r] 
&Q_B \ar[r] \ar@<-.4ex>[u]_{0}
&Q_A \ar[r] \ar@<-.4ex>[u]_{0}
&Q_D \ar[r] \ar@<-.4ex>[u]_{0}
&0
},
\end{align*}
we also have 
$\cl{P_A}=\cl{P_B}+\cl{P_D}$ and 
$\cl{Q_A}=\cl{Q_B}+\cl{Q_D}$.
Using these relations we easily see 
\begin{align*}
&K_{-\cl{P_B}} * K_{\cl{D}} = K_{-\cl{P_A}} * K_{\cl{Q_D}},
\\
&K_{-\cl{P_A}} * K_{\cl{P_B}} = K_{-\cl{P_D}}.
\end{align*}
Moreover we can compute
\begin{align*}
  \<P_A,A\>+\<D,P_B\>-\<Q_D,B\>-(P_A,D)
&=\<P_A,B\>+\<P_A,D\>+\<D,P_B\>-\<D,B\>-\<P_D,B\>-(P_A,D)
\\
&=\<P_A,B\>+\<D,P_B\>-\<D,B\>-\<P_D,B\>-\<D,P_A\>
\\
&=\<P_B,B\>+\<D,P_B\>-\<D,B\>-\<D,P_B\>-\<D,P_D\>
\\
&=\<P_B,B\>-\<D,P_D\>-\<D,B\>.
\end{align*}
Therefore \eqref{eq:embed:L} and \eqref{eq:embed:R} coincide.
\end{proof}

%%%%%%%%%%%%%%%%%%%%%%%%%%%%%%%%%%%%%%%%%%%%%%%%%%%
%%%%%%%%%%%%%%%%%%%%%%%%%%%%%%%%%%%%%%%%%%%%%%%%%%%

\subsection{Bialgebra structure of Bridgeland's Hall algebra}

We continue to consider the product  on 
$\cal{DH}(\cal{A}) \ctp \cal{DH}(\cal{A})$
defined as  
\begin{align}
\label{eq:mul:tensor}
 \bigl([\bu{M}] \otimes [\bu{N}]\bigr) * 
 \bigl([\bu{P}] \otimes [\bu{Q}]\bigr)
 := \bigl([\bu{M}] * [\bu{N}]\bigr) 
    \otimes \bigl([\bu{P}] * [\bu{Q}]\bigr)
\end{align}
and the map
$$
\chi_0(\bu{M},\bu{N}) := -\<\bu{N},\bu{M}\>.
$$

\begin{thm}
\label{thm:main1}
For an abelian category $\cal{A}$ satisfying 
the conditions (a)--(e),
the tuple $(\cal{DH}(\cal{A}),*,[0],\Delta_{\chi_0},\ep)$ 
is a topological bialgebra
under the multiplication \eqref{eq:mul:tensor} 
on $\cal{DH}(\cal{A}) \ctp \cal{DH}(\cal{A})$.
\end{thm}

\begin{proof}
By the construction of $\Delta_{\chi_0}$ 
(Definition~\ref{dfn:coprod}),
it is enough to show 
$\Delta_{\chi_0}(E_A * E_B) =\Delta_{\chi_0}(E_A) * \Delta_{\chi_0}(E_B)$.
But it follows from Fact~\ref{fct:B:lem4.6} and Theorem~\ref{thm:embedding},
so we are done.
\end{proof}

Combining the result of \cite{Go} and \cite{Y}, we also have

\begin{thm}
For an abelian category $\cal{A}$ satisfying 
the conditions (a)--(e),
the tuple $(\cal{DH}_{\red}(\cal{A}),*,[0],\Delta_{\chi_0},\ep)$ 
coincides with the Drinfeld double of 
$(\cal{H}^{\e}_{\tw}(\cal{A}),*,[0],\Delta_{\cal{H}^{\e}_{\tw}(\cal{A})},\ep)$ 
as a bialgebra.
\end{thm}

%%%%%%%%%%%%%%%%%%%%%%%%%%%%%%%%%%%%%%%%%%%%%%%%%%
%%%%%%%%%%%%%%%%%%%%%%%%%%%%%%%%%%%%%%%%%%%%%%%%%%
%%%%%%%%%%%%%%%%%%%%%%%%%%%%%%%%%%%%%%%%%%%%%%%%%%

\end{document}